\documentclass{amsart}
\usepackage{amsfonts}
\usepackage[active]{srcltx}
\usepackage{amsmath}
\usepackage{amssymb}
\usepackage{amsthm}
\usepackage{mathtools}
\usepackage{shuffle}
\usepackage{bm} 
\usepackage{wrapfig}

\usepackage{enumitem}
\usepackage{wasysym}
\usepackage{mathscinet}
\usepackage{verbatim}
\usepackage{graphicx}
\usepackage[
bookmarks=true,         
bookmarksnumbered=true, 
colorlinks=true, pdfstartview=FitV, linkcolor=blue, citecolor=blue,
urlcolor=blue]{hyperref}
\usepackage{microtype}
\usepackage{comment}
\usepackage{xcolor,cancel}
\usepackage[normalem]{ulem}
\usepackage{amsmath}
\newcommand{\stkout}[1]{\ifmmode\text{\sout{\ensuremath{#1}}}\else\sout{#1}\fi}
\usepackage{xcolor,cancel}

\theoremstyle{plain}
\newtheorem{theorem}{Theorem}[section]

\newtheorem{lemma}[theorem]{Lemma}
\newtheorem{problem}[theorem]{Problem}

\renewenvironment{proof}{{\indent \indent \it Proof:l\quad}}{ \hfill$\blacksquare$\par}
\theoremstyle{remark}

\theoremstyle{definition}
\newtheorem{remark}[theorem]{Remark}
\newtheorem{definition}[theorem]{Definition}

\def\Om{{\Omega}}
\def\om{{\omega}}

\def\mb{\mathbb}
\def\om{\omega}
\def\Om{\Omega}
\def\ra{\rightarrow}

\DeclareMathOperator*{\Div}{\mathrm{div}}

\def\mb{\mathbb}
\def\om{\omega}
\def\Om{\Omega}
\def\ra{\rightarrow}
\def\df{{\rm d}}
\def\mcH{\mathcal{H}}
\def\pt{\partial}

\let\Section=\section
\def\section{\setcounter{equation}{0}\Section}

\title[Carleman estimates of degenerated     parabolic equations]{ Carleman estimates for degenerate parabolic equations
	with single interior point degeneracy and its applications}
\date{April 2024}
\author[W. Wu, Y.Hu, Y. Liu, D. Yang]{Weijia Wu$^{1}$, Yaozhong Hu$^{2}$, Yuanhang Liu$^{1,*}$,   Donghui Yang$^{1}$ }
\thanks{${}^*$Corresponding author: liuyuanhang97@163.com}
\address{$^1$ School of Mathematics and Statistics, Central South University, Changsha 410083, China} 
\address{$^2$  Department of Mathematical and Statistical Sciences, University of Alberta, Edmonton, AB T6G 2G1, Canada} 
\email{weijiawu@yeah.net}
\email{yaozhong@ualberta.ca}
\email{liuyuanhang97@163.com}
\email{donghyang@outlook.com}

\begin{document}
	\begin{abstract} 
		We study the controllability of a class of $N$-dimensional degenerate parabolic equations with single interior point degeneracy. We employ the Galerkin method to prove the existence of solutions for the equations. The analysis is then divided into two cases based on whether the degenerate point $x=0$ lies within the control region $\omega_0$ or not. For each case, we establish specific Carleman estimates. As a result, we achieve null controllability in the first case $0\in\omega_0$ and unique continuation and approximate controllability in the second case $0\notin\omega_0$.
	\end{abstract}
	\pagestyle{myheadings}
	\thispagestyle{plain}
	\markboth{DEGENERATE PARABOLIC EQUATIONS WITH GENERAL SYMMETRIC COEFFICIENTS}{}
	
	\maketitle		
	
	\section{Introduction}

	
	In recent years, scholars have been paying a lot of attention to degenerate equations. Indeed, numerous problems in physics (see \cite{calin2009heat,calin2010heat}), biology (see \cite{martinez2003regional}), and finance (see \cite{sakthivel2008exact}) can be described by degenerate parabolic equations.
	Unlike traditional parabolic equations, degenerate parabolic equations do not meet the uniform ellipticity condition. This distinction suggests additional challenges in establishing the existence and uniqueness of solutions. Moreover, the degeneracy of coefficients further complicates the proof of controllability.
	Unique continuation is a powerful tool for analyzing controllability.
	When it comes to parabolic equations, it's interesting to note that there are essentially two main methods for proving the unique continuation property. One method is described in \cite{garofalo1986monotonicity,garofalo1987unique} based on a combination of geometric and variational ideas, while the other involves using the Carleman estimate.
	In the world of math analysis and partial differential equations, the Carleman estimate is a really important tool. It has unique properties that make it super useful for solving all kinds of math problems.	
	It's widely recognized that Carleman estimates are key tools for proving unique continuation and controllability. 
	
	Carleman estimates for uniformly parabolic operators without degeneracy have been extensively developed (see \cite{alabau2012carleman,carleman1939probleme,dolecki1977general,emanuilov1995controllability,fattorini1971exact,fattorini1974uniform,fursikov1996controllability,lebeau1995controle,russell1973unified,vancostenoble2008null}). In more recent times, scholars have explored Carleman estimates for degenerate operators.  
	In the context of one-dimensional equations, boundary degeneracy has been explored in works such as \cite{araujo2022boundary,CA7,flores2020null}. Likewise, investigations into interior degeneracy have been conducted, as evidenced by studies like \cite{cannarsa2019null,fragnelli2016interior,fragnelli2017carleman}.
	The study in \cite{cannarsa2019null} focused on the case of interior degeneracy in one-dimensional equations, where the degeneracy occurs at the interior point $x = 0$, and locally distributed controls act only on one side of the origin.
	The authors utilized spectral analysis and the moment method to establish null controllability for the case when $\alpha \in (0,1)$. Unlike \cite{cannarsa2019null}, our study considers a similar case of interior single-point degeneracy in high-dimensional degenerate parabolic equations. The key distinction from \cite{cannarsa2019null} is that our study focuses on the interior single-point degeneracy of degenerate parabolic equations in higher dimensions. We consider Dirichlet boundary conditions and utilize the Carleman estimate method to establish null controllability for the case where $0 \in \omega_0$, and approximate controllability for the case where $0 \notin \omega_0$. Here, $\omega_0$ denotes the control region we have chosen.

	The research on high-dimensional degenerate parabolic equations is currently relatively limited (see \cite{BSV,cannarsa2016global,stuart2020critically,stuart2018stability}).
	In \cite{stuart2018stability}, the authors studied the stability of solutions to high-dimensional interior single-point degenerate equations.
	As for the study of controllability in the context of high-dimensional equations with boundary degeneracy, current investigations are primarily focused on cases where control is applied at the boundary but not in the interior (see \cite{BSV,cannarsa2016global}). 
	Notably, there is a lack of research addressing the controllability of high-dimensional equations with interior degeneracy when control is applied within the interior.
	
	Importantly, it should be emphasized that there exists a significant distinction between interior degeneracy and boundary degeneracy. When degeneracy
	occurs on the boundary, the influence of boundary conditions becomes pronounced. In \cite{wu2023null}, the equation studies the equation that degenerated on partial  boundary, while the control system studied in this paper degenerated at a single interior point. degenerated on partial  boundary means that the regularity of our boundary cannot be improved, and it also means that the divergence theorem of the boundary term may not hold in the computation of Carleman estimates. Therefore, in \cite{wu2023null}, we set the weight of a small region containing degenerate boundaries to 0, so that the integral term on this region is equal to 0. This is also the key idea of our method in \cite{wu2023null}: for some difficult to handle integral terms, we replace the integral over the degenerate region with the integral over the non-degenerate region for estimation.
	In this paper, our research differs significantly from that in \cite{wu2023null}. In \cite{wu2023null}, the control region is selected on a partial boundary and includes the entire degenerate boundary. However, in cases where the control region does not contain degenerate points $(0 \notin \om_0)$, we no longer use this method.
	Instead, for each solution of the dual equation, we introduce a sufficiently small $\epsilon$,  such that the integral terms over the degenerate region can be estimated by the integral terms over the non-degenerate region.
	This makes the coefficients $C$ in the final Carleman estimate dependent on the solution of dual equation. Therefore, in the second case, we cannot obtain the observability inequality and can only obtain the 
	unique continuation. Therefore, we only prove the approximate controllability in the second case, which is not be studied in \cite{wu2023null}.
	%
	
	In this paper, we study the Carleman estimates of a class of high-dimensional degenerate parabolic equations with interior single-point degeneracy. We have established Carleman estimates for both cases: $0 \in \omega_0$ and $0 \notin \omega_0$. For the case where $0 \in \omega_0$, we adopt the approach used in \cite{wu2023null}  to derive our Carleman estimate, thus establishing null controllability. On the other hand, for the case where $0 \notin \omega_0$, we employ a similar method as in the standard parabolic case, but with a different weight function. As a result, we obtain unique continuation and approximate controllability.
	
	The remaining sections of this paper are structured as follows. In Section 2, we present the main results of this paper. In Section 3, we present some preliminary results and demonstrate the well-posedness of problem \eqref{1.1}. In Section 4, we divide into two cases, denoted as $0 \in \omega_0$ and $0 \notin \omega_0$ respectively, and provide Carleman estimates for the degenerate equation \eqref{1.1}. 
	
	\section{main results}
	Consider the following single interior point degenerate problem
	\begin{equation}\label{1.1}
		\begin{cases}
			\partial_{t}z-\Div(|x|^\alpha A\nabla z)=\chi_{\omega_0}g, & (x,t)\in \Omega\times(0,T),\\
			z(x,t)=0, & (x,t)\in \partial\Omega\times(0,T),\\
			z(0)=z_{0},  &x\in\Omega.
		\end{cases}
	\end{equation}
	The solution of the equation \eqref{1.1} is denoted as $z(t;g,z_0)$. $\Omega$ is a bounded open subset of  $\mathbb{R}^N$ ($N\ge 2$) with $0\in\Omega$ and boundary $\Gamma:= \partial\Omega$  of class $C^2$, $0<\alpha<2$, $T>0$, $ Q:=\Omega\times (0,T)$, $\Sigma:= \Gamma\times(0,T)$.
	Consider $\epsilon_0\in (0, d(0,\Gamma))$ with $d(0,\Gamma)=\inf_{y\in \Gamma} |y|$, and for a small constant $\epsilon\in ( 0,\frac{1}{9}\epsilon_0) $, let
	\begin{equation*}
		\Om^\epsilon = \left\lbrace x\in \Om \mid |x|>\epsilon \right\rbrace, \ \Om_\epsilon = \left\lbrace x\in \Om \mid |x|<\epsilon \right\rbrace \mbox{ and } S(\epsilon) = \left\lbrace x\in \mathbb{R}^N \bigm| |x|=\epsilon \right\rbrace,
	\end{equation*}
	$\omega_0 \subset \Omega$ is a nonempty open set, $\om$ is a nonempty open set that satisfies $\overline{\om} \subset \om_0$ and $\Om_{2\epsilon}\cap\overline{\om}=\emptyset$. $\chi_{\omega_0}$ is the corresponding characteristic function, the symmetric matrix-valued function $A:\overline{\Omega} \to M_{n\times n}(\mathbb{R})$  satisfies the uniform ellipticity condition, i.e., $\xi^T A \xi \ge \beta |\xi|^2$,  for all $\xi \in \mathbb{R}^N$, where $\beta>0$ is a constant, and $A_{ij}\in C^1(\overline{\Omega})$ with $A=(A_{ij})$,  $i,j=1,\cdots, N$. $g\in L^2(Q)$ is the control, and $z_0 \in L^{2}(\Omega)$ is the initial data.
	
	We denote
	\begin{equation*}
		\mathcal{A}z=\Div(|x|^\alpha A\nabla z). 
	\end{equation*}
	
	The solution space is
	\begin{equation*}
		\mcH_0^1(\Om)=\left\lbrace z\in W_0^{1,1}(\Om)\bigm| |x|^{\alpha } A \nabla z \cdot \nabla z \in L^1(\Om)\right\rbrace,
	\end{equation*}
	where
	\begin{equation*}
		\left\langle z,v\right\rangle _{\mcH_0^1(\Om)}= \int_\Om |x|^\alpha A\nabla z\cdot \nabla v dx,
	\end{equation*}
	and
	\begin{equation*}\label{norm}
		\left\| z\right\|^{2} _{\mcH_0^1(\Om)}= \int_\Om |x|^\alpha A\nabla z\cdot \nabla z dx.
	\end{equation*}
	Using the classic Galerkin method, we will prove the following well posedness results in Section 3.
	\begin{theorem}\label{exist1}
		For any $g \in L^2(Q)$ and any $z_0 \in L^2(\Omega)$, there exists a unique solution $z \in C^0([0,T];L^2(\Omega)) \cap L^2(0,T;\mathcal{H}_0^1(\Omega))$ to equation \eqref{1.1}. Moreover, there exists a positive constant $C$ such that
		$$
		\sup_{t\in\left[0,T\right] } \|z(t)\|_{L^2(Q)}^2 + \int_{0}^{T} \left\| z(t)\right\|^2_{\mcH_0^1(\Omega)} d t \le C(\left\| z_0\right\|_{L^2(\Om)}^2 + \|g\|_{L^2(Q)}^2).
		$$
	\end{theorem}
	
	
	
	
	As is classical in controllability problems, we can introduce the so-called adjoint problem for \eqref{1.1}:
	\begin{equation}\label{3.1}
		\begin{cases}
			\partial_{t}w + \Div(|x|^\alpha A\nabla w)=f, & \mbox{in} \ Q,  \\
			w=0 ,  & \mbox{on} \ \Sigma,  \\
			w(x,  T)= w_T,  & \mbox{in} \ \Omega,  
		\end{cases}
	\end{equation}
	where $f\in L^2(Q)$ and $w_T \in L^2(\Omega)$. The main results of this paper are the Carleman inequality and observability inequality stated below.
	\subsection{Null controllability for the case of \texorpdfstring{$0\in \om_0$}{}}
	
	\hspace*{\fill}\\
	
	For the case of $0\in\omega_0$,  we will establish the following Carleman estimate, with the detailed proof provided in Subsection 4.1.
	\begin{theorem}\label{Carleman1}
		There exist positive constants $C=C(\om_0,\Om)$, $s_0, \lambda_0$ such that for any $\lambda \ge \lambda_0$, $s \ge s_0$, and for any solution $w$ to \eqref{3.1}, the following inequality holds:
		\begin{equation}\label{3.2}
			\begin{split}
				&s^{-1}\iint_{Q} \xi^{-1} (|w_t |^2 + \left|\Div(|x|^\alpha A\nabla w)\right|^2) dx dt \\
				&+C s^3 \lambda^4 \iint_Q\xi^3|w|^2 dx dt  + C s  \lambda^2 \iint_Q\xi\left||x|^\alpha A  \nabla \eta \cdot \nabla w \right|^2 dx  d t\\
				&\qquad\le C\left\|e^{-s \sigma} f\right\|_{L^2(Q)}^2
				+Cs^3 \lambda^4 \int_0^T \int_{\omega_{0}} \xi^3|w|^2dx  d t.
			\end{split}
		\end{equation}
	\end{theorem}
	Then, employing a standard argument  (see \cite{FE,fursikov1996controllability}), we can deduce the resulting observability inequality as follows.
	\begin{theorem}
		For a fixed $T>0$ and an open set $\omega_0 \subset \Omega$ as defined previously, assuming that \eqref{3.2} holds, there exists a positive constant $C>0$ such that for any $w_T \in L^2(\Omega)$, the solution to \eqref{3.1} satisfies
		\begin{equation*}\label{3.4}
			\int_{\Omega}|w(x, 0)|^2 d x \leq C \iint_{\omega_0 \times(0, T)}|w|^2 d x d t.
		\end{equation*}
	\end{theorem}
	Using the duality between controllability and observability, proving controllability is equivalent to establishing an observability inequality for the adjoint system \eqref{3.1} (see \cite{lions1992remarks,rockafellar1967duality}). From this, we can deduce the null controllability of \eqref{1.1}.
	\begin{theorem}
		For a fixed $T>0$ and an open set $\omega_0 \subset \Omega$ as defined previously, assuming that \eqref{3.2} holds, there exists a control $g \in L^2(Q)$ such that the solution $z$ of \eqref{1.1} satisfies
		$$
		z(\cdot, T)=0, \text { in } \Omega.
		$$
	\end{theorem}
	\subsection{Unique Continuation and approximate controllability for the case of \texorpdfstring{$0\notin \om_0$}{}}
	
	\hspace*{\fill}\\
	
	For the case of $0\notin \om_0$, we have the following Carleman estimate results, with the detailed proof provided in Subsection 4.2.
	\begin{theorem}\label{Carleman2}
		For any solution $w$ to \eqref{3.1}, there exist positive constants $C=C(\om_0,\Om,u)$, $s_0, \lambda_0$ such that for any $\lambda \ge \lambda_0$, $s \ge s_0$, the following inequality holds:
		\begin{equation}\label{08.16.1}
			\begin{split}
				&Cs^{-1} \iint_Q \xi^{-1}\left|w_t\right|^2dx  d t +Cs^{-1}   \iint_Q \xi^{-1}\left|\Div(|x|^\alpha A\nabla w)\right|^2dx  d t\\
				&+Cs^3\lambda^4\iint_Q \xi^3  w^2 dxdt
				+Cs \lambda^2 \iint_Q \xi  \left| |x|^\alpha A \nabla  w \cdot \nabla \eta\right| ^2  dx  dt\\
				&+C s\lambda^2 \iint_Q \xi   |x|^\alpha A \nabla w \cdot \nabla w dx  dt\\
				&\qquad\le \left\| e^{-s\sigma} f \right\| ^2_{L^2(Q)}
				+ Cs^3\lambda^4\int_{0}^{T}\int_{\om_0} \xi^3 w^2 dxdt.
			\end{split}
		\end{equation}
	\end{theorem}
	Then, using a standard argument (see \cite{FE,fursikov1996controllability}), we can easily conclude the following unique continuation results.
	\begin{theorem}
		For a fixed $T>0$ and an open set $\omega_0 \subset \Omega$ as defined previously, assuming that \eqref{08.16.1} holds, then for any $w_T \in L^2(\Omega)$,  if $w=0$ in $\om_0\times(0,T)$, then $w=0$ in $Q$.
	\end{theorem}
	Now that we have obtained the unique continuation of equation \eqref{3.1}, we can easily obtain the approximate controllability of equation \eqref{1.1} as follows.
	\begin{theorem}
		For a fixed $T>0$ and an open set $\omega_0 \subset \Omega$ as defined previously, assuming that \eqref{08.16.1} holds, then the system \eqref{1.1} is approximate controllability, i.e. 
		$$
		\overline{\left\lbrace z(T;g,z_0) \mid g\in L^2(Q)\right\rbrace }=L^2(\Om).
		$$
	\end{theorem}
	\begin{remark}
		It is worth noting that in case $0\notin \om_0$, we can only achieve approximate controllability. This is because the constant $C$ in \eqref{08.16.1} depends on $u$. Consequently, it is hard to directly derive the observability inequality from \eqref{08.16.1}. We will provide a more detailed explanation of this in Section 4.2.
	\end{remark}

	\section{Well-posed results}
	In this section we discuss the well-posedness of \eqref{1.1}. To begin, we shall present some preliminary results. 
	\subsection{Some preliminaries}
	
	\hspace*{\fill}\\
	
	Let us initially present some fundamental lemmas.
	\begin{lemma}\label{L2.2}
		For any $N\ge 2$ and $\alpha\in (0,2)$, we have $|x|^{\frac{\alpha}{2}-1} z \in L^2(\Om)$ for all $z\in \mcH_0^1(\Om)$ with 
		\begin{equation}\label{2.1}
			(N-2+\alpha) \left\| |x|^{\frac{\alpha}{2}-1}  z \right\| _{L^2(\Om)} \le C\left\| z \right\| _{\mcH_0^1(\Om)}, 
		\end{equation}
		where $C>0$ is a constant that is depending on $z\in \mathcal{H}_0^1(\Om)$. 
		Furthermore, $z\in L^2(\Om)$ if $z\in \mathcal{H}_0^1(\Om)$. 
	\end{lemma}
	\begin{proof}
		If $z \in \mcH_0^1(\Om)$ its restriction belongs to $W^{1,2}\left(\Om^{\epsilon}\right)$ and  thus its trace represents a bounded linear map into $L^2\left(\partial \Om^{\epsilon}\right)$. 
		Then
		$$
		\begin{aligned}
			2 \int_{\Om^{\epsilon}} |x|^{\alpha-2}  z(x \cdot \nabla z) d x & =\int_{\Om^{\epsilon}} |x|^{\alpha-2}  x \cdot \nabla\left(z^2\right) d x \\
			& =-\int_{S(\epsilon)} |x|^{\alpha-1}  z^2 d s-\int_{\Om^{\epsilon}}(N-2+\alpha) |x|^{\alpha-2}  z^2 d x
		\end{aligned}
		$$
		since the trace of $z$ is zero on $\partial \Omega$, we have
		$$
		\begin{aligned}
			(N-2+\alpha) \int_{\Om^{\epsilon}} |x|^{\alpha-2}  z^2 d x & \le-2 \int_{\Om^{\epsilon}} |x|^{\alpha-2}  z(x \cdot \nabla z) d x \le 2 \int_{\Om^{\epsilon}}\left(|x|^{\frac{\alpha}{2}-1}  |z|\right)\left(|x|^{\frac{\alpha}{2}}|\nabla z|\right) d x \\
			& \le 2\left\{\int_{\Om^{\epsilon}} |x|^{\alpha-2} z^2 d x\right\}^{1 / 2}\left\{\int_{\Om^{\epsilon}} |x|^\alpha   \nabla z \cdot \nabla z  d x\right\}^{1 / 2}\\
			&\leq C\left\{\int_{\Om^{\epsilon}} |x|^{\alpha-2} z^2 d x\right\}^{1 / 2}\left\{\int_{\Om^{\epsilon}} |x|^\alpha   A\nabla z \cdot \nabla z  d x\right\}^{1 / 2}
		\end{aligned}
		$$
		and \eqref{2.1} follows by letting $\epsilon \rightarrow 0^+$ since $|x|^\alpha A\nabla z\cdot\nabla z \in L^1(\Omega)$.
	\end{proof}		
	\begin{lemma}\label{HS}	
		The space $(\mcH_0^1(\Omega),\left\langle \cdot,\cdot\right\rangle _{\mcH_0^1(\Omega)})$ is a Hilbert space.   
	\end{lemma}
	
	\begin{proof}	
		Firstly, we easily verify that $(\mcH_0^1(\Omega),\left\langle \cdot,\cdot\right\rangle_{\mcH_0^1(\Omega)})$ is an inner space. 			
		Next let us prove
		$(\mcH_0^1(\Omega),\left\langle \cdot,\cdot\right\rangle_{\mcH_0^1(\Omega)})$ is a Hilbert space.  
		For simplicity, we define the weighted space
		\begin{equation*}
			L_\alpha^2(\Om):=\left\lbrace w\left|\ \! w: \Om\ra \mb{R}^N \mbox{ is a measurable function}, \mbox{ and } \int_{\Om} |x|^\alpha A w\cdot w dx  < \infty\right. \right\rbrace.
		\end{equation*}
		It is easily checked that $L_\alpha^2(\Om)$ is a Hilbert space with the inner product
		\begin{equation*}
			(w_1,w_2)=\int_\Om |x|^\alpha Aw_1\cdot w_2dx=\int_\Om \left(|x|^\frac{\alpha}{2}A^\frac{1}{2}w_1\right)\cdot\left(|x|^\frac{\alpha}{2}A^\frac{1}{2}w_2\right)dx. 
		\end{equation*}
		
		Let $\{v_n\}_{n\in\mb{N}}\subset\mcH_0^1(\Omega)$ be a Cauchy sequence. Then there exist $v\in L^2(\Omega)$, $g=(g_1,g_2,\dots,g_N)\in L^2(\Omega)^N$ such that 
		\begin{equation*}
			v_n\ra v \mbox{ strongly in } L^2(\Om), \mbox{and } |x|^{\frac{\alpha}{2}}A^{\frac{1}{2}}\nabla v_n \ra g \mbox{ strongly in } L^2(\Omega)^N. 
		\end{equation*}
		Therefore, if we show that $\nabla v = |x|^{-\frac{\alpha}{2}}A^{-\frac{1}{2}}g$, we obtain the 
		conclusion. For this purpose, since $v_n\ra v \mbox{ strongly in } L^2(\Omega)$, one has $\frac{\partial v_n}{\partial x_i}\rightarrow \frac{\partial v}{\partial x_i}\ (i=1,2,\dots N)$ in the sense of distribution 
		and the distributional limit is unique. So, it is sufficient to prove that
		\begin{equation*}
			\nabla v_n\ra |x|^{-\frac{\alpha}{2}}A^{-\frac{1}{2}}g  \mbox{ strongly in } L_\alpha^2(\Om) \Rightarrow \nabla v_n\ra  |x|^{-\frac{\alpha}{2}}A^{-\frac{1}{2}}g \mbox{ strongly in } \mathcal{D}'(\Om)
		\end{equation*}
		according to $|x|^{\frac{\alpha}{2}}A^{\frac{1}{2}}\nabla v_n \ra g \mbox{ strongly in } L^2(\Omega)^N$, where $\mathcal{D}'(\Om)$ represents the dual space of $C_0^{\infty}(\Omega)$. Indeed, 
		for all $\varphi\in C_0^{\infty}(\Omega)$, one has that
		\begin{equation*}
			\begin{split}
				&\left|\int_\Om (\nabla v_n  - |x|^{-\frac{\alpha}{2}}A^{-\frac{1}{2}}g)\varphi dx \right|\\ 
				&\qquad\le \left\| \varphi\right\|_{L^\infty (\Om)}  \int_{\Om} \left| \nabla v_n  - |x|^{-\frac{\alpha}{2}}A^{-\frac{1}{2}}g \right| dx\\
				&\qquad\le\|\varphi\|_{L^\infty(\Om)}\int_\Om \left||x|^{-\frac{\alpha}{2}}A^{-\frac{1}{2}}\right| \left||x|^\frac{\alpha}{2}A^\frac{1}{2}\nabla v_n-g\right|dx\\
				&\qquad\le C\|\varphi\|_{L^\infty(\Om)}\left(\int_\Om |x|^{-\alpha}d x\right)^\frac{1}{2}\left(\int_\Om \left||x|^\frac{\alpha}{2}A^\frac{1}{2}\nabla v_n-g\right|^2dx\right)^\frac{1}{2}\\
				&\qquad\le C\|\varphi\|_{L^\infty(\Om)}\left\|\nabla v_n-|x|^{-\frac{\alpha}{2}}A^{-\frac{1}{2}}g\right\|_{L_\alpha^2(\Om)}\rightarrow 0, \ n\rightarrow\infty,
			\end{split}
		\end{equation*}
		here and in what follows, we denote $C$ the different constants by the context.
	\end{proof}

	\begin{lemma}\label{dense}
		$C_0^{\infty}(\Omega)$ is dense in $\mcH_0^{1}(\Omega)$.
	\end{lemma}
	
	
	\begin{proof}
		To demonstrate the density of $C_0^{\infty}(\Omega)$ in $\mcH_0^1(\Omega)$,  choose $\xi \in C^\infty(\mathbb{R}^N)$ such that $0\le \xi \le 1$ for all $x\in \mathbb{R}^N$ with $\xi(x)=1$ for $|x|\ge 2$ and $\xi(x)=0$ for $|x|\le 1$. 
		
		Consider integers $n\ge \frac{2}{\epsilon}$ where $\epsilon \in (0,\epsilon_0)$ and set $\xi_n (x)=\xi (nx)$. Note that $ | \nabla \xi_n (x) | \le n \left\| \nabla \xi \right\| _{L^\infty}$ for all $x$. Then for any $z\in \mathcal{H}_0^1(\Om)$, $z_n=\xi_n z$ belongs to $H_0^1 (\Om)$ and
		\begin{equation*}
			\begin{split}
				\left\| z-z_n \right\|^2_{\mcH_0^1(\Om)}=&\int_{\Om}  |x|^\alpha A  \nabla \left( (1-\xi_n) z\right) \cdot \nabla \left( (1-\xi_n) z\right) dx\\
				\le& 2\int_{\Om}  \left(|x|^\alpha A \nabla \xi_n \cdot \nabla \xi_n  z^2 + (1-\xi_n)^2 A\nabla z \cdot \nabla z \right)  dx.
			\end{split}
		\end{equation*}
		But
		\begin{equation*}
			\begin{split}
				\int_{\Om} |x|^\alpha A \nabla \xi_n \cdot \nabla \xi_n  z^2 dx &= \int_{\frac{1}{n}<|x|<\frac{2}{n}} |x|^\alpha A \nabla \xi_n \cdot \nabla \xi_n  z^2 dx\\
				&\le Cn^2 \left\| A^{\frac{1}{2}} \nabla \xi \right\| ^2_{L^\infty} \int_{\frac{1}{n}<|x|<\frac{2}{n}} \left( \frac{2}{n}\right) ^2 |x|^{\alpha-2}  z^2 dx\\
				&\le C \left\| A^{\frac{1}{2}} \nabla \xi \right\| ^2_{L^\infty} \int_{\frac{1}{n}<|x|<\frac{2}{n}} |x|^{\alpha-2}  z^2 dx ,
			\end{split}
		\end{equation*}
		which implies
		\begin{equation*}
			\int_{\Om} |x|^\alpha A \nabla \xi_n \cdot \nabla \xi_n  z^2 dx\ra 0 \mbox{ as } n\ra \infty,
		\end{equation*}
		where 
		\begin{equation*}
			\lim\limits_{n\to \infty} \int_{\frac{1}{n}<|x|<\frac{2}{n}}  |x|^{\alpha-2}  z^2 dx =0 
		\end{equation*}
		since   $|x|^{\alpha-2}  z^2 \in L^1 (\Om)$ from \eqref{2.1}.
		Moreover,
		\begin{equation*}
			\begin{split}
				&\int_{\Om} |x|^\alpha A (1-\xi_n)^2 \nabla z \cdot \nabla z dx =\int_{|x|\le \frac{2}{n}} |x|^\alpha A  (1-\xi_n)^2 \nabla z \cdot \nabla z dx \le \int_{|x|\le \frac{2}{n}} |x|^\alpha A   \nabla z \cdot \nabla z dx,
			\end{split}
		\end{equation*}
		which implies
		\begin{equation*}
			\int_{\Om} |x|^\alpha A (1-\xi_n)^2 \nabla z \cdot \nabla z dx \mbox{ as } n\rightarrow 0, 
		\end{equation*}
		where
		\begin{equation*}
			\lim\limits_{n\to \infty} \int_{|x|\le \frac{2}{n}}   |x|^{\alpha} A \nabla z \cdot \nabla z dx =0
		\end{equation*}
		since $\int_{\Om} |x|^{\alpha} A \nabla z \cdot \nabla z dx < \infty$. 
		Hence $\xi_n z \to z$ in $\mathcal{H}_0^1(\Om)$, showing that $H_0^1(\Om)$ is dense in $\mcH_0^1(\Om)$.
		But $C_0^{\infty}(\Omega)$ is dense in $H_0^1(\Om)$ with its Dirichlet norm
		\begin{equation*}
			\left\| z \right\|_{H_0^1(\Om)} = \left( \int_{\Om} |\nabla z |^2 dx \right) ^{\frac{1}{2}},
		\end{equation*}
		and, since $\Om$ is bounded, there exists a constant $D$ such that
		\begin{equation*}
			\|z\|_{\mathcal{H}_0^1(\Om)}=\int_{\Om} |x|^\alpha A  \nabla z \cdot \nabla z dx \le D \int_{\Om} |\nabla z |^2 dx=D\|z\|_{H_0^1(\Om)} \mbox{ for all } z\in H_0^1(\Om). 
		\end{equation*}
		This proves that $C_0^{\infty}(\Omega)$ is dense in $\mcH_0^1(\Om)$.
	\end{proof}

	\begin{theorem}\label{compactembed}
		$\mcH_0^1(\Omega)$ is compactly embeded in $L^2(\Omega)$.
	\end{theorem}
	\begin{proof}
		To establish the compactness of the embedding it is suffices to show that if $\left\{z_n\right\}$ is a sequence converging weakly to zero in $\mcH_0^1(\Om)$ as $n \rightarrow \infty$, then $\left\|z_n\right\|_{L^2(\Om)} \rightarrow 0$ as $n \rightarrow \infty$  by abstract subsequence.
		
		Since $\mcH_0^1(\Om)$ is continuously embedded in $L^2(\Omega)$ by Lemma \ref{L2.2}, $L^2(\Omega)^* \subset$ $\mcH_0^1(\Om)^*$ and hence $\left\{z_n\right\}$ converges weakly to zero in $L^2(\Omega)$. 
		
		Consider $\epsilon \in\left(0, \epsilon_0\right]$. If $\left\{z_n\right\}$ does not converge weakly to zero in $W^{1,2}\left(\Om^{\epsilon}\right)$, there exist $f \in W^{1,2}\left(\Om^{\epsilon}\right)^*$, a subsequence $\left\{z_{n_k}\right\}$ and $\delta>0$ such that $\left|f\left(z_{n_k}\right)\right| \geq \delta$ for all $n_k$. Passing to a further subsequence if necessary, we can suppose that $\left\{z_{n_k}\right\}$ converges weakly to an element $v$ in $W^{1,2}\left(\Om^{\epsilon}\right)$. Thus $\left\{z_{n_k}\right\}$ converges weakly to $v$ in $L^2\left(\Om^{\epsilon}\right)$ and so $v=0$ a.e. on $\Om^{\epsilon}$ since $\left\{z_n\right\}$ converges weakly to zero in $L^2(\Omega)$ and hence also on $L^2\left(\Om^{\epsilon}\right)$. But then $f\left(z_{n_k}\right) \rightarrow f(v)=f(0)=0$ as $n_k \rightarrow \infty$, contradicting the choice of $\delta$. Hence $\left\{z_n\right\}$ converges weakly to zero in $W^{1,2}\left(\Om^{\epsilon}\right)$ and therefore $\left\|z_n\right\|_{L^2\left(\Om^{\epsilon}\right)} \rightarrow 0$ as $n \rightarrow \infty$ . It follows that
		\begin{equation}\label{UN}
			\limsup _{n \rightarrow \infty}\left\|z_n\right\|_{L^2(\Om)}^2=\limsup _{n \rightarrow \infty} \int_{B(0, \epsilon)}\left|z_n\right|^2 d x.
		\end{equation}
		But from \cite{catrina2001caffarelli} and Lemma \ref{L2.2}, we have
		\begin{equation}\label{LP}
			\left\|z \right\|_{L^q(\Om)} \le C \left\| z\right\|_{\mcH_0^1(\Om)}, \ 1\le q \le \frac{2N}{N-2+\alpha},
		\end{equation}
		then (taking $q>2$)
		\begin{equation*}
			\int_{B(0, \epsilon)}\left|z_n\right|^2 d x \le \left(\int_{B(0, \epsilon)} |1|^ {\frac{q}{q-2}} dx\right) ^{\frac{q-2}{q}} \left( \int_{B(0, \epsilon)} \left( |z_n|^2 \right) ^{\frac{q}{2}}dx\right)^{\frac{2}{q}}
			\le |B(0, \epsilon)|^ {\frac{q-2}{q}} \left\| z_n \right\| ^2_{L^q(\Om)}.
		\end{equation*}
		The weak convergence of $\left\{z_n\right\}$ in $\mcH_0^1(\Om)$ and \eqref{LP} imply that this sequence is bounded in $L^q(\Omega)$, then (note that $q>2$)
		$$
		\int_{B(0, \epsilon)}\left|z_n\right|^2 d x \le C |B(0,\epsilon)|^\frac{q-2}{q}\left\| z_n\right\|^{2}_{\mcH_0^1(B(0, \epsilon))}\leq C |B(0,\epsilon)|^\frac{q-2}{q}.
		$$
		Letting $\epsilon \rightarrow 0^+$ in \eqref{UN} shows that $\left\|z_n\right\|_{L^2(\Om)} \rightarrow 0$ as $n \rightarrow \infty$, completing the proof.
	\end{proof}
	
	\subsection{Proof of Theorem \ref{exist1}}
	
	\hspace*{\fill}\\
	
	Here, we will use the classical Galerkin method to prove Theorem \ref{exist1}. Since we have obtained that the solution space $\mcH_0^1(\Omega)$ is compactly embeded in $L^2(\Omega)$ in Lemma \ref{compactembed}, we can easily prove the existence of the solution for \eqref{1.1}.
	First, let us give the definition of weak solutions.
	\begin{definition}
		We say a function
		\begin{equation*}
			z \in L^{2}(0,T;\mathcal{H}_0^{1}(\Omega)), \ \mbox{with} \  \frac{d}{dt}z \in L^{2}(0,T;\mathcal{H}^{-1}(\Omega)),
		\end{equation*}
		is a {\it weak solution} of the degenerate parabolic problem \eqref{1.1} provided
		\begin{equation*}
			\left\langle \frac{d}{dt}z,v\right\rangle_{\left\langle \mathcal{H}^{-1}(\Omega),\mathcal{H}_0^{1}(\Omega)\right\rangle  } + B \left[ z,v;t\right] =(\chi_{\omega_0}g,v)
		\end{equation*}
		for each $v\in \mathcal{H}_0^{1}(\Omega)$ and a.e.  $0\le t \le T$ and
		\begin{equation*}
			z(0)=z_0,
		\end{equation*}
		where 
		\begin{equation*}
			B[z, v;t]=\int_{\Omega}|x|^\alpha A\nabla z\cdot \nabla v  dx.  
		\end{equation*}
	\end{definition}
	Since $\mcH_0^1(\Omega)$ is compactly embedded in $L^2(\Omega)$ by Theorem \ref{compactembed}, we have $\mathcal{A}^{-1}$ is a compact and symmetric operator.  By Theorem 7 in D.6 in \cite{Evans}, there exists a countable orthogonal basis of  $\mathcal{H}_0^1(\Omega)$ consisting of eigenvectors of $\mathcal{A}^{-1}$.  Let $\left\lbrace w_k \right\rbrace _{k=1}^{\infty}$ be such a basis, forms an orthonormal basis in both $\mathcal{H}_0^1(\Omega)$ and $L^2(\Omega)$. Following a similar argument to 7.1.2 in \cite{Evans},
	for a fixed positive integer $m$, we can find a function  $ z_{m}:[0, T] \to \mcH_0^1(\Omega) $ of the form
	\begin{equation}\label{a5}
		z_{m}(t):=\sum_{k=1}^{m}d_{m}^{k}(t)w_{k},
	\end{equation}
	that satisfies the following equations: 
	\begin{equation}\label{a7}
		\left(  \frac{d}{dt}z_{m}, w_{k} \right)  +B[z_{m}, w_{k};t] =(\chi_{\omega_0}g, w_{k}),   \quad k=1, \cdots, m,
	\end{equation}
	for $  t\in (0, T)$ a.e., here, $\left( \cdot , \cdot \right) $ denotes the inner product in $L^2(\Om)$, and
	\begin{equation*}\label{a8}
		B[z_{m}, w_{k};t]=\int_{\Omega}|x|^\alpha A\nabla z_{m}\cdot \nabla  w_{k} dx, 
	\end{equation*}
	and $ d_{m}^{k}(t)\ (k=1, . . . , m) $ are the desired coefficients such that
	\begin{equation}\label{a6}
		d_{m}^{k}(0)=(z_{0},  w_{k}),  \quad k=1, \cdots,  m. 
	\end{equation} 
	
	Thus we seek a function  $ z_{m} $ of the form \eqref{a5} that satisfies \eqref{a7} of problem \eqref{1.1} onto the finite dimensional subspace spanned by  $ \{w_{k}\}^{m}_{k=1}$. 
	Next, using the classic Galerkin method in 7.1.2 in \cite{Evans}, we can easily prove the following results. (The details of the proof of Theorem \ref{b1}, Theorem \ref{b2}, and Theorem \ref{b3} are in the Appendix.)

	\begin{theorem}[Construction of approximate solutions]\label{b1}
		For each integer $m\in\mathbb{N}$,   there exists a unique function  $ z_{m}$ of the form \eqref{a5} satisfying \eqref{a7} and \eqref{a6}. 
	\end{theorem}
	
	%

	\begin{theorem}[Energy estimates]\label{b2}
		There exists a constant C,  depending only on  $ \Omega $,  and the coefficients of $\mathcal{A}$,  such that
		\begin{equation*}\label{a12}
			\begin{split}
				&\max\limits_{t\in(0, T)} \|z_{m}(t) \|_{L^{2}(\Omega)}+\|z_{m} \|_{L^{2}(0, T;\mcH_0^1(\Omega))}+\left\|\frac{d}{dt}z_{m}\right\|_{L^{2}(0, T;\mathcal{H}^{-1}(\Omega))} \\
				&\qquad\le C \left(\|\chi_{\omega_0}g \|_{L^{2}(0, T;L^{2}(\Omega))}+\|z_{0} \|_{L^{2}(\Omega)}\right),           
			\end{split}
		\end{equation*}
		where $\mathcal{H}^{-1}(\Omega)$ is the dual space of $\mathcal{H}_0^1(\Omega)$. 
	\end{theorem}
	%
	%
	%
	

	\begin{theorem}[Existense of the weak solution]\label{b3}
		There exists a  weak solution of \eqref{1.1} and the weak solution is unique.  
	\end{theorem}
	From Theorem \ref{b1}, Theorem \ref{b2}, and Theorem \ref{b3}, we can easily deduce the conclusion of Theorem \ref{exist1}.
	Moreover, we can achieve improved regularity by making difference with respect to  $x$ (see 6.3. (page 327) in \cite{Evans})
	and applying the trace theorem (see 5.5. (page 271) in \cite{Evans}). 
	\begin{theorem}[Improving Regularity] \label{regularity}
		For all solution $w$ of equation \eqref{3.1}, we have $w=w(t,x) \in H^1(0,T; \mathcal{H}_0^1(\Omega)) \cap  L^2(0,T; H^2(\Om^\epsilon))$. Moreover, $\nabla w\in [L^2(\Sigma)]^N$.
	\end{theorem}

	\section{Carleman esitimates}
	In order to obtain controllability results, we need to make some Carleman estimates. We will divide it into two cases and establish two different Carleman estimates accordingly.
	\subsection{The case of \texorpdfstring{$0 \in \omega_0$}{}}
	
	\hspace*{\fill}\\
	
	In this Subsection, we mainly introduce the Carleman estimation in the case of $0 \in \omega_0$. Before proving Theorem \ref{Carleman1}, we first provide a lemma about the weight function. The setting of the weight function is important, as it uses the weight function to eliminate the influence of interior degeneracy. Here, we assume $\Om_{6\epsilon}\subset \om_0$.		
	\begin{lemma}\label{ETA}
		There is a weight function that satisfies the following conditions:
		\begin{itemize}
			\item [($i$)] 
			\begin{equation*}
				\eta(x) :=
				\begin{cases}
					=0, & x \in \Omega_\epsilon\cup \partial\Om,\\
					>0, & x \in \Omega^\epsilon,
				\end{cases}
			\end{equation*}
			\item [($ii$)] 
			\begin{equation*}
				|x|^\alpha A \nabla \eta \cdot \nabla \eta \ge C >0 \ \mbox{in} \ \Om\backslash\omega_{0}, \quad  \nabla \eta=0 \ \mbox{on} \ \overline{\Omega_\epsilon}.
			\end{equation*}
			\item [($iii$)] $|x|^\alpha A \nabla \eta$, $|x|^\alpha A \nabla \eta \cdot \nabla \eta$, $\nabla(|x|^\alpha A \nabla \eta \cdot \nabla \eta)$, $D(|x|^\alpha A \nabla \eta)$,  $\Div(|x|^\alpha A \nabla \eta)$, $\nabla(|x|^\alpha A \nabla \eta\cdot \nabla \eta)$, $\nabla(\Div(|x|^\alpha A \nabla \eta))$, $\left( |x|^\alpha A \nabla  \eta\right)_i \frac{\partial(|x|^\alpha A)}{\partial x_i} \in C^0(\overline{{\Omega}})$.
		\end{itemize}
		Moreover, we have
		$\frac{\partial \eta}{\partial {n}} <0 \mbox{ on } \Gamma$, 
		where ${n}$ is the unit outer normal vector on $\Gamma=\partial\Om$.
	\end{lemma}
	\begin{proof}
		Consider a Morse function $\zeta \in C^3(\overline{\Omega})$ satisfying $\zeta = 0$ on $\partial \Om$, $\zeta>0$ in $\Om$,  and $\nabla \zeta \ge C>0$ in $\Om\backslash\omega_{0}$. Define
		\begin{equation*}
			h(x) :=
			\begin{cases}
				h=0, & x \in \Omega_\epsilon,\\
				h=1, & x \in \Om\backslash\Omega_{2\epsilon},\\
				0<h<1, & x \in \Omega_{2\epsilon}\backslash\Omega_\epsilon.
			\end{cases}\
			\mbox{ and } h \in C^\infty(\Om).
		\end{equation*}
		Let $\eta=h\cdot \zeta$, then $\eta$ satisfies item ($i$) and item ($ii$) since $\Omega_{6\epsilon}\subset \omega_{0}$.
		
		Finally, we have $\frac{\partial \eta}{\partial {n}} <0 \mbox{ on } \Gamma$ since $\eta>0$ in $\Om$ and $\eta=0$ on $\Gamma$.
	\end{proof}
	
	Now, for any $\lambda>\lambda_0$ and $s>0$, let us introduce the following functions and
	constants:
	\begin{equation*}
		\begin{split}
			& \theta(t):=[t(T-t)]^{-4}, \quad \xi(x, t):=\theta(t) e^{ \lambda(8|\eta|_\infty+\eta (x))}, \quad \sigma(x, t):=\theta(t) e^{10 \lambda|\eta|_\infty}-\xi(x, t).
		\end{split}
	\end{equation*}
	In what follows, $C>0$ represents a generic constant, and $w$ denotes a solution of equation \eqref{3.1}. We can assume, using standard arguments, that $w$ possesses sufficient regularity. Specifically, we have $w \in H^1(0,T; \mathcal{H}_0^1(\Omega))\cap L^2(0,T; H^2(\Om^\epsilon)$.
	
	Next, let us establish the Carleman estimate for case $0 \in \omega_0$.
	
	{\bf Proof of Theorem \ref{Carleman1}:}
	Consider $s>s_0>0$ and introduce
	\begin{equation*}
		u=e^{-s\sigma} w.  
	\end{equation*}
	Then, the following properties hold for $u$:
	\begin{itemize}
		\item [($i$)] $u=\frac{\partial u}{\partial x_i}=0$ at $t=0$ and $t=T$;
		\item [($ii$)] $u=0$  on $\Sigma$;
		\item [($iii$)] Let 
		\begin{equation*}
			\begin{split}
				P_1 u
				&:=u_t+s \Div(u |x|^\alpha A \nabla \sigma)+s |x|^\alpha A\nabla \sigma \cdot \nabla u,\\
				P_2 u
				&:=\Div(|x|^\alpha A \nabla u)+s^2 u |x|^\alpha A \nabla \sigma \cdot \nabla \sigma+s \sigma_t u.
			\end{split}
		\end{equation*}
		Then 
		\begin{equation*}
			P_1 u+P_2 u=e^{-s \sigma} f.
		\end{equation*}
	\end{itemize}
	
	From ($iii$), it follows that
	\begin{equation}\label{PP}
		\left\|P_1 u\right\|^2+\left\|P_2 u\right\|^2+2\left(P_1 u, P_2 u\right)=\left\|e^{-s \sigma} f\right\|^2.
	\end{equation}
	
	Let us define $\left(P_1 u, P_2 u\right)=I_1+\cdots+I_4$, where
	\begin{equation*}
		\begin{split}
			& I_1:=\left(\Div(|x|^\alpha A \nabla u)+s^2 u |x|^\alpha A \nabla \sigma \cdot  \nabla \sigma+s \sigma_t u, u_t\right), \\
			& I_2:=s^2\left(\sigma_t u, \Div(u |x|^\alpha A \nabla \sigma)+|x|^\alpha A \nabla \sigma \cdot  \nabla u\right), \\
			&I_3:=s^3\left( u |x|^\alpha A \nabla \sigma \cdot  \nabla \sigma, \Div(u |x|^\alpha A \nabla \sigma)+\nabla \sigma \cdot |x|^\alpha A \nabla u\right),\\
			& I_4:=s(\Div(|x|^\alpha A \nabla u), \Div(u |x|^\alpha A \nabla \sigma)+\nabla \sigma \cdot |x|^\alpha A \nabla u).
		\end{split}
	\end{equation*}
	By differentiating with respect to the time variable $t$, since $u_t \in \mathcal{H}_0^1(\Omega)$, by the definition of $u$ and properties ($i$) and ($ii$), we have
	\begin{equation}\label{I11}
		\begin{split}
			I_1=&\iint_Q u_t \Div (|x|^\alpha A\nabla u) +s^2 u |x|^\alpha A \nabla \sigma \cdot \nabla \sigma u_t + s\sigma_t u u_t dx  dt\\
			=&\int_{\Om} s^2  \left(|x|^\alpha A \nabla \sigma \cdot \nabla \sigma\right) \frac{1}{2}u^2 dx  \bigg|_0^T + \int_{\Om} s\sigma_t\cdot \frac{1}{2}u^2 dx  \bigg|_0^T+ \iint_\Sigma u_t |x|^\alpha A\nabla u \cdot n ds dt\\
			&-\iint_{Q}  |x|^\alpha A\nabla u \cdot \nabla u_t dx  dt -\frac{1}{2}\iint_{Q}\left(s\sigma_t + s^2 |x|^\alpha A \nabla \sigma \cdot \nabla \sigma\right)_t u^2 dx  dt\\
			=& -\frac{1}{2}\iint_{Q}(s\sigma_t + s^2 |x|^\alpha A \nabla \sigma \cdot \nabla \sigma)_t u^2 dx  dt.
		\end{split}
	\end{equation}
	Furthermore, we obtain
	\begin{equation}\label{I22}
		\begin{split}
			I_2  =& s^2  \iint_Q \sigma_t u (\Div(u |x|^\alpha A \nabla \sigma) + |x|^\alpha A\nabla u \cdot \nabla \sigma) dx  dt\\
			=&  s^2  \iint_Q \Div( |x|^\alpha A \nabla \sigma) \sigma_t u^2 dx  dt + s^2 \iint_\Sigma \sigma_t u^2 |x|^\alpha A \nabla \sigma  \cdot n dsdt\\
			&- s^2  \iint_Q \Div(\sigma_t |x|^\alpha A \nabla \sigma)u^2 dx dt \\
			=& -s^2\iint_Q  |x|^\alpha A \nabla \sigma \cdot\nabla \sigma_t u^2    dx  dt.
		\end{split}
	\end{equation}
	In the second equality, we can observe that $s^2 \iint_\Sigma \sigma_t u^2 |x|^\alpha A \nabla \sigma \cdot n dsdt=0$ since $u \in \mathcal{H}_0^1(\Omega)$.
	
	Similarly, in $I_3$ below, $s^3 \iint_\Sigma u^2 (|x|^\alpha A \nabla \sigma \cdot \nabla \sigma) |x|^\alpha A \nabla \sigma \cdot n ds dt=0$  due to $u \in \mathcal{H}_0^1(\Omega)$.
	In (\ref{I33}), we have
	\begin{equation}\label{I33}
		\begin{split}
			I_3  =&s^3 \iint_Q u |x|^\alpha A \nabla \sigma \cdot \nabla \sigma (\Div (u |x|^\alpha A \nabla \sigma) + |x|^\alpha A\nabla u \cdot \nabla \sigma ) dx  dt\\
			=&s^3 \iint_\Sigma u^2 (|x|^\alpha A \nabla \sigma \cdot \nabla \sigma) |x|^\alpha A \nabla \sigma \cdot n  ds dt -s^3 \iint_Q u |x|^\alpha A \nabla \sigma \cdot \nabla (u |x|^\alpha A \nabla \sigma \cdot \nabla \sigma  ) dx  dt\\
			&+ s^3 \iint_Q  (|x|^\alpha A \nabla \sigma \cdot \nabla u) ( |x|^\alpha A \nabla \sigma \cdot \nabla \sigma  ) u dx  dt\\
			=&-s^3\iint_Q  |x|^\alpha A\nabla \sigma \cdot \nabla ( |x|^\alpha A\nabla \sigma \cdot\nabla \sigma) u^2 dx  dt.
		\end{split}
	\end{equation}
	Similarly, for $I_4$, we have
	\begin{equation*}
		\begin{split}
			I_4= & s \iint_Q \Div(|x|^\alpha A \nabla u)\left(  \Div(u |x|^\alpha A \nabla \sigma)+ |x|^\alpha A \nabla u \cdot \nabla \sigma\right)  dx  dt \\
			= & s \iint_\Sigma u \Div(|x|^\alpha A \nabla \sigma )|x|^\alpha A \nabla u \cdot n dsdt - 2s \lambda \iint_\Sigma  \xi\left[|x|^\alpha A\nabla u\cdot n\right]^2\frac{\partial \eta}{\pt n} ds dt \\
			&-s \iint_Q |x|^\alpha A \nabla u \cdot  \nabla\left( u \Div(|x|^\alpha A \nabla \sigma)\right) dx  dt -2s \iint_{Q} |x|^\alpha A \nabla u\cdot \nabla \left( |x|^\alpha A \nabla u \cdot \nabla \sigma \right)  dx  dt,
		\end{split}
	\end{equation*}
	Since
	\begin{equation*}
		\begin{split}
			&-2s \iint_{Q} |x|^\alpha A \nabla u \cdot \nabla (|x|^\alpha A\nabla u \cdot \nabla \sigma) dx dt\\
			&\qquad=-2s \sum_{i=1}^{N} \iint_{Q} (|x|^\alpha A\nabla \sigma)_i |x|^\alpha A \nabla u \cdot \frac{\partial}{\partial x_i}(\nabla u) + (\nabla u)_i |x|^\alpha A \nabla u \cdot \nabla (|x|^\alpha A \nabla \sigma)_i dx dt,\\
		\end{split}
	\end{equation*}
	where
	\begin{equation*}
		\begin{split}
			&-2s \sum_{i=1}^{N} \iint_{Q} (|x|^\alpha A\nabla \sigma)_i |x|^\alpha A \nabla u \cdot \frac{\partial}{\partial x_i}(\nabla u) dx dt\\
			&\qquad=s \lambda \iint_\Sigma  \xi(|x|^\alpha An\cdot n)(|x|^\alpha A\nabla u\cdot \nabla u)\frac{\partial\eta}{\partial n}ds dt + s\iint_{Q} (|x|^\alpha A \nabla u \cdot \nabla u ) \Div(|x|^\alpha A\nabla \sigma) dx dt\\
			&\qquad\qquad+s \sum_{i=1}^{N} \iint_{Q} (|x|^\alpha A\nabla \sigma)_i \frac{\partial (|x|^\alpha A)}{\partial x_i} \nabla u \cdot \nabla u dx dt,
		\end{split}
	\end{equation*}
	Hence,  note that $\nabla u=\frac{\partial u}{\partial n}n$ on $\Sigma$, we have
	\begin{equation}\label{I44}
		\begin{split}
			I_4 = & -s \iint_Q u |x|^\alpha A \nabla u\cdot \nabla (\Div(|x|^\alpha A \nabla \sigma))dx  dt
			-2s  \sum_{i=1}^{N}\iint_{Q} (\nabla u)_i |x|^\alpha A \nabla u \cdot \nabla (|x|^\alpha A \nabla \sigma)_i dx dt  \\
			&
			+ s  \sum_{i=1}^{N}\iint_{Q} (|x|^\alpha A\nabla \sigma)_i \frac{\partial (|x|^\alpha A)}{\partial x_i} \nabla u \cdot \nabla u dx dt- s \lambda \iint_\Sigma  \xi\left( |x|^\alpha A\nabla u\cdot n\right) ^2\frac{\partial \eta}{\pt n} ds dt.
		\end{split}
	\end{equation}
	By combining \eqref{I11} to \eqref{I44},we can conclude that
	\begin{equation}\label{PPu}
		\begin{aligned}
			\left(P_1 u, P_2 u\right)
			=& -s^3 \iint_Q |x|^\alpha A \nabla \sigma\cdot \nabla(|x|^\alpha A \nabla \sigma \cdot \nabla \sigma)|u|^2 dx  dt\\
			&-2s  \sum_{i=1}^{N}\iint_{Q} (\nabla u)_i |x|^\alpha A \nabla u \cdot \nabla (|x|^\alpha A \nabla \sigma)_i dx dt
			-2s^2 \iint_Q |x|^\alpha A \nabla \sigma \cdot \nabla \sigma_t|u|^2 dx  dt\\
			&+ s \sum_{i=1}^{N} \iint_{Q} (|x|^\alpha A\nabla \sigma)_i \frac{\partial (|x|^\alpha A)}{\partial x_i} \nabla u \cdot \nabla u dx dt
			-s \iint_Q u |x|^\alpha A \nabla u \cdot\nabla(\Div(|x|^\alpha A \nabla \sigma)) dx dt\\
			&-\frac{s}{2} \iint_Q \sigma_{t t}|u|^2 dx  dt- s \lambda \iint_\Sigma  \xi\left[|x|^\alpha A\nabla u\cdot n\right]^2\frac{\partial \eta}{\pt n} ds dt.
		\end{aligned}
	\end{equation}
	Let us denote the seven integrals on the right-hand side of \eqref{PPu} by $T_1, \cdots, T_7$. Now, we will estimate each of them.
	Using the definitions of $\sigma$ and $\xi$ and the properties of $\eta$, we have 
	\begin{equation*}
		\nabla \sigma = - \lambda \xi \nabla \eta,\quad \nabla \xi = \lambda \xi \nabla \eta,\quad |x|^\alpha A \nabla \sigma \cdot \nabla \sigma = \lambda^2 \xi^2 |x|^\alpha A\nabla \eta \cdot \nabla \eta,
	\end{equation*}
	and 
	\begin{equation*}
		\nabla( |x|^\alpha A \nabla \sigma \cdot \nabla \sigma) = \lambda^2 \xi^2 \nabla (|x|^\alpha A\nabla \eta \cdot \nabla \eta) + 2\lambda^3 \xi^2 \nabla \eta \left(  |x|^\alpha A \nabla \eta \cdot \nabla \eta \right). 
	\end{equation*}
	Therefore, we can rewrite $T_1$ as follows: 
	\begin{equation}
		\begin{aligned}
			T_1
			&=-s^3 \iint_Q |x|^\alpha A \nabla \sigma \cdot \nabla(|x|^\alpha A \nabla \sigma \cdot \nabla \sigma)|u|^2 dx  dt\\
			&=2s^3 \lambda^4 \iint_Q \xi^3\left||x|^\alpha A \nabla \eta \cdot \nabla \eta \right|^2|u|^2 dx  dt
			+ s^3 \lambda^3 \iint_Q \xi^3 |x|^\alpha A\nabla \eta \cdot \nabla(|x|^\alpha A\nabla \eta \cdot \nabla\eta)|u|^2 dx  dt.
		\end{aligned}
	\end{equation}
	Since $|x|^\alpha A\nabla \eta \cdot\nabla(|x|^\alpha A\nabla \eta \cdot \nabla\eta)$ is bounded in $\overline{\Omega}$, we have
	$$
	s^3 \lambda^3 \int_{0}^{T}\int_{\omega_{0}} \xi^3 |x|^\alpha A\nabla \eta \cdot \nabla(|x|^\alpha A\nabla \eta \cdot \nabla\eta)|u|^2 dx  dt \ge -C s^3 \lambda^3 \int_{0}^{T}\int_{\omega_{0}} \xi^3 |u|^2 dx  dt,
	$$
	and $||x|^\alpha A \nabla \eta \cdot \nabla \eta|\ge C>0$ in $\overline{\Omega} \setminus \omega_{0}$. Thus, we have
	$$
	\begin{aligned}
		&s^3 \lambda^3 \int_{0}^{T}\int_{\Om\backslash \omega_{0}} \xi^3 |x|^\alpha A\nabla \eta \cdot \nabla(|x|^\alpha A\nabla \eta \cdot \nabla\eta)|u|^2 dx  dt\\
		&\qquad\ge -C s^3 \lambda^3 \int_{0}^{T}\int_{\Om\backslash \omega_{0}} \xi^3 |u|^2 dx  dt
		\ge
		-C s^3 \lambda^3 \int_{0}^{T}\int_{\Om\backslash \omega_{0}} \xi^3 \left||x|^\alpha A \nabla \eta \cdot \nabla \eta \right|^2 |u|^2 dx  dt\\
		&\qquad\ge
		-C s^3 \lambda^3 \iint_{Q} \xi^3 \left||x|^\alpha A \nabla \eta \cdot \nabla \eta \right|^2 |u|^2 dx  dt.
	\end{aligned}
	$$
	Hence, we can rewrite $T_1$ as (for $\lambda>0$ large enough)
	\begin{equation}\label{T1}
		\begin{aligned}
			T_1
			\geq C s^3 \lambda^4 \iint_Q \xi^3\left||x|^\alpha A \nabla \eta \cdot \nabla \eta \right|^2|u|^2 dx  dt-C s^3 \lambda^3\int_0^T \int_{\omega_{0}} \xi^3|u|^2 dx  dt.
		\end{aligned}
	\end{equation}
	Similarly, for $T_2$, we have
	\begin{equation}\label{T2}
		\begin{aligned}
			T_2 =&-2s \sum_{i=1}^{N} \iint_{Q} (\nabla u)_i |x|^\alpha A \nabla u \cdot \nabla (|x|^\alpha A \nabla \sigma)_i dx dt\\
			=&2s \lambda^2\iint_{Q}  \xi \left|  |x|^\alpha A\nabla \eta \cdot \nabla u \right| ^2 dx  dt  + 2s\lambda \iint_{Q}\xi D(|x|^\alpha A\nabla \eta) |x|^\alpha A\nabla u  \cdot \nabla u dx  dt\\
			\ge& Cs \lambda^2\iint_{Q}  \xi \left|  |x|^\alpha A\nabla \eta \cdot \nabla u \right| ^2 dx  dt  - C s\lambda \iint_{Q}\xi |x|^\alpha A\nabla u  \cdot \nabla u dx  dt.
		\end{aligned}
	\end{equation}
	Furthermore, for $T_3$, as can be seen from the definition of $\xi$, we have $\xi \xi_t \le \xi^3$, thus
	\begin{equation*}
		\begin{aligned}
			T_3 &=-2 s^2 \iint_{Q} |x|^\alpha A\nabla \sigma \cdot \nabla \sigma_t u^2 dx dt
			=-2 s^2 \lambda^2 \iint_{Q} \xi \xi_t \left||x|^\alpha A \nabla \eta \cdot \nabla \eta \right|  u^2 dx dt\\
			&\ge-2 s^2 \lambda^2 \iint_{Q} \xi^3 \left||x|^\alpha A \nabla \eta \cdot \nabla \eta \right|  u^2 dx dt.
		\end{aligned}
	\end{equation*}
	Similar to $T_1$, we have
	\begin{equation}\label{T3}
		\begin{aligned}
			T_3
			\ge &-C s^2 \lambda^2 \iint_{Q} \xi^3 \left||x|^\alpha A \nabla \eta \cdot \nabla \eta \right|^2  u^2 dx dt -C s^2 \lambda^2 \int_{0}^{T}\int_{\omega_{0}} \xi^3   u^2 dx dt,
		\end{aligned}
	\end{equation}
	and for $T_4$, we have
	\begin{equation}\label{T4}
		\begin{aligned}
			T_4 
			=&-s\lambda \sum_{i=1}^N\int_0^T\int_{\Om^\epsilon}\xi (|x|^\alpha A\nabla \eta)_i\frac{\partial(|x|^\alpha A)}{\partial x_i}\nabla u\cdot \nabla udxdt\geq -Cs\lambda \int_0^T\int_{\Om^\epsilon} \xi |\nabla u|^2\df x\df t\\
			\ge&  -Cs\lambda \int_0^T\int_{\Om^\epsilon} \xi |x|^\alpha A\nabla u\cdot \nabla udxdt \ge -C s \lambda \iint_Q  \xi |x|^\alpha A \nabla u \cdot \nabla u dx  dt.
		\end{aligned}
	\end{equation}
	Furthermore, by utilizing the definitions of $\sigma$ and $\xi$, we obtain
	\begin{equation}\label{T5}
		\begin{split}
			T_5  =&-s \iint_Q u |x|^\alpha A \nabla u \cdot \nabla(\Div(|x|^\alpha A \nabla \sigma)) dx  dt\\
			=&s \lambda^3 \iint_Q \xi u |x|^\alpha A \nabla u \cdot \nabla \eta \left( |x|^\alpha A \nabla \eta \cdot \nabla \eta\right)  dx  dt
			+ s \lambda^2 \iint_Q \xi u |x|^\alpha A \nabla u \cdot \nabla\left(|x|^\alpha A \nabla \eta \cdot \nabla \eta\right) dx  dt \\
			& +s \lambda^2 \iint_Q \xi u |x|^\alpha A \nabla u \cdot \nabla \eta \Div(|x|^\alpha A \nabla \eta)   dx  dt+s \lambda \iint_Q \xi u |x|^\alpha A \nabla u \cdot \nabla(\Div(|x|^\alpha A \nabla \eta)) dx  dt.
		\end{split}
	\end{equation}
	Let us denote $T_{51}, \cdots, T_{54}$ as the seven integrals on the right-hand side of \eqref{T5}. Then we have
	\begin{equation}\label{T51}
		\begin{split}
			T_{51}=&s \lambda^3 \iint_Q \xi u |x|^\alpha A \nabla u \cdot \nabla \eta \left( |x|^\alpha A \nabla \eta \cdot \nabla \eta\right)  dx  dt\\
			\geq&
			-Cs^2 \lambda^4 \iint_Q \xi\left||x|^\alpha A \nabla \eta \cdot \nabla \eta\right|^2|u|^2 dx  dt
			-C\lambda^2 \iint_Q \xi| |x|^\alpha A\nabla u \cdot\nabla \eta|^2 dx  dt,
		\end{split}
	\end{equation}
	and (by the definition of $\eta$)
	\begin{equation}\label{T52}
		\begin{split}
			T_{52}=&s \lambda^2 \iint_Q \xi u |x|^\alpha A \nabla u \cdot \nabla\left(|x|^\alpha A \nabla \eta \cdot \nabla \eta\right) dx  dt\\
			\ge&-Cs^2 \lambda^3\iint_{Q} \xi \left| |x|^\alpha A \nabla \eta \cdot \nabla \eta \right| u^2 dx  dt -Cs^2 \lambda^3\int_0^T \int_{\omega_{0}} \xi   u^2 dx  dt\\
			&- C\lambda \iint_Q  \xi |x|^\alpha A \nabla u \cdot \nabla u dx  dt.
		\end{split}
	\end{equation}
	Then we have
	\begin{equation}\label{T53}
		\begin{split}
			T_{53}=&s \lambda^2 \iint_Q \xi u |x|^\alpha A \nabla u \cdot \nabla \eta  \Div(|x|^\alpha A \nabla \eta) dx  dt\\
			\ge& -Cs^2 \lambda^3\iint_Q \xi \left| |x|^\alpha A \nabla \eta \cdot \nabla \eta \right|^2  u^2  dx  dt -Cs^2 \lambda^3\int_0^T \int_{\omega_{0}} \xi u^2 dx  dt\\
			&- C\lambda \iint_Q \xi |x|^\alpha A \nabla u \cdot \nabla u dx  dt,
		\end{split}
	\end{equation}
	and
	\begin{equation}\label{T54}
		\begin{split}
			T_{54}=&s \lambda \iint_Q \xi u |x|^\alpha A \nabla u \cdot\nabla(\Div(|x|^\alpha A \nabla \eta))dx  d t\\
			\ge&
			-C s^2 \lambda \iint_Q \xi \left||x|^\alpha A \nabla \eta \cdot \nabla \eta\right|^2 u^2dx  d t-C s^2 \lambda \int_0^T \int_{\omega_{0}}\xi  u^2dx  d t\\
			&-C \lambda \iint_Q \xi |x|^\alpha A \nabla u \cdot \nabla udx  d t.
		\end{split}
	\end{equation}
	Combining \eqref{T5} to \eqref{T54}, we obtain
	\begin{equation}\label{T55}
		\begin{split}
			T_5 \geq&-Cs^2 \lambda^4 \iint_Q \xi\left||x|^\alpha A \nabla \eta \cdot \nabla \eta\right|^2 |u|^2dx  d t -C\lambda^2\iint_{Q} \xi \left| |x|^\alpha A\nabla u\cdot \nabla \eta \right| ^2 dx  dt\\
			&-C s^2 \lambda^3 \int_0^T \int_{\omega_{0}} \xi|u|^2dx  d t- C\lambda \iint_Q  \xi |x|^\alpha A \nabla u \cdot \nabla u dx  d t.
		\end{split}
	\end{equation}
	Finally, as we can observe from the definitions of $\xi$ and $\sigma$, we have $\sigma_{tt} \leq \xi^{\frac{3}{2}}$. Hence,
	\begin{equation}\label{T6}
		T_6 =-\frac{s}{2} \iint_Q \sigma_{t t}|u|^2dx  d t \geq-C s \iint_Q \xi^{3 / 2}|u|^2dx  d t.
	\end{equation}
	and by Lemma \ref{ETA}, we have
	\begin{equation}\label{T7}
		T_7 =- s \lambda \iint_\Sigma  \xi\left[|x|^\alpha A\nabla u\cdot n\right]^2\frac{\partial \eta}{\pt n} ds dt\ge 0.
	\end{equation}
	From equations \eqref{T1}-\eqref{T4}, \eqref{T55}, and \eqref{T6}-\eqref{T7}, we deduce the following inequality:
	\begin{equation}\label{PPU}
		\begin{split}
			\left(P_1 u, P_2 u\right)
			\geq& Cs^3 \lambda^4 \iint_Q \xi^3 |\nabla \eta \cdot |x|^\alpha A \nabla \eta|^2 |u|^2 dx dt  + C s  \lambda^2 \iint_Q\xi\left||x|^\alpha A  \nabla \eta \cdot \nabla u \right|^2 dx  d t\\
			&-Cs^3 \lambda^3 \int_0^T \int_{\omega_{0}} \xi^3|u|^2dx  d t
			- Cs \lambda \iint_Q \xi |x|^\alpha A \nabla u \cdot \nabla udx  d t\\
			&- Cs \iint_Q \xi^{3 / 2}|u|^2dx  d t.
		\end{split}
	\end{equation}
	since
	\begin{equation*}
		\begin{split}
			&C s \lambda^2 \iint_Q \xi^{3 / 2}|u|^2  dx  d t\\
			&\qquad= C s \lambda^2 \int_0^T \int_{\Om\backslash \om} \xi^{3 / 2}|u|^2  dx  d t + C s \lambda^2 \int_0^T \int_{\om} \xi^{3 / 2}|u|^2  dx  d t\\
			&\qquad\le C s \lambda^2 \int_0^T \int_{\Om\backslash \om} \left| |x|^\alpha A \nabla \eta \cdot \nabla \eta \right|^2 \xi^{3 / 2}|u|^2  dx  d t + C s \lambda^2 \int_0^T \int_{\om} \xi^{3 / 2}|u|^2  dx  d t\\
			&\qquad\le C s \lambda^2 \iint_Q \left| |x|^\alpha A \nabla \eta \cdot \nabla \eta \right|^2 \xi^{3 / 2}|u|^2  dx  d t + C s \lambda^2 \int_0^T \int_{\om} \xi^{3 / 2}|u|^2  dx  d t,
		\end{split}
	\end{equation*}
	clearly, $s \iint_Q \xi^{3 / 2}|u|^2dx dt$ can be absorbed by other terms. Combining \eqref{PP} and \eqref{PPU}, we can conclude that
	\begin{equation}\label{LLU2}
		\begin{split}
			&\left\|P_1 u\right\|^2+\left\|P_2 u\right\|^2+C\iint_Q s^3 \lambda^4 \xi^3\left||x|^\alpha A \nabla \eta \cdot \nabla \eta \right|^2|u|^2 dx dt\\
			&+ C s  \lambda^2 \iint_Q\xi\left||x|^\alpha A  \nabla \eta \cdot \nabla u \right|^2 dx  d t\\
			&\qquad\le C\left\|e^{-s \sigma} f\right\|^2
			+Cs \lambda \iint_Q \xi |x|^\alpha A \nabla u \cdot \nabla udx  d t+Cs^3 \lambda^3 \int_0^T \int_{\omega_{0}} \xi^3|u|^2dx  d t.
		\end{split}
	\end{equation}
	Using the definitions of $P_1 u$ and $P_2 u$, we can observe that
	\begin{equation*}
		\begin{split}
			&s^{-1} \iint_Q \xi^{-1}\left|u_t\right|^2dx  d t\\
			&\qquad= s^{-1} \iint_Q \xi^{-1}(P_1 u -s u \Div(|x|^\alpha A\nabla \sigma)- 2s \nabla u \cdot |x|^\alpha A\nabla \sigma)^2dx  d t\\
			&\qquad\le Cs^{-1}\left\|P_1 u\right\|^2+s   \iint_Q \xi^{-1}|u|^2|\Div(|x|^\alpha A \nabla \sigma)|^2dx  d t +C s \iint_Q \xi^{-1}|\nabla u |x|^\alpha A \nabla \sigma|^2dx  d t \\
			&\qquad\le Cs^{-1}\left\|P_1 u\right\|^2+C s \lambda^4 \iint_Q \xi\left| |x|^\alpha A \nabla \eta \cdot \nabla \eta \right|^2 |u|^2dx  d t \\
			&\qquad\qquad +C s \lambda^2 \iint_Q \xi|u|^2  dx  d t+C s \lambda^2  \iint_Q \xi\left||x|^\alpha A  \nabla \eta \cdot \nabla u \right|^2dx  d t.
		\end{split}
	\end{equation*}
	Obviously, $C s \lambda^2\iint_Q \xi|u|^2dx dt$ can be absorbed by other terms.
	\begin{equation*}
		\begin{split}
			&s^{-1}   \iint_Q \xi^{-1}\left|\Div(|x|^\alpha A\nabla u)\right|^2dx  d t\\
			&\qquad= s^{-1}   \iint_Q \xi^{-1}(P_2 u -s^2 u |x|^\alpha A \nabla \sigma \cdot  \nabla \sigma - s \sigma_t u)^2dx  d t\\
			&\qquad\le C s^{-1} \left\|P_2 u\right\|^2+C s^3 \lambda^4   \iint_Q \xi^3||x|^\alpha A \nabla \eta\cdot \nabla \eta|^2|u|^2dx  d t +C s \lambda^2  \iint_Q \xi^{2}|u|^2dx  d t.
		\end{split}
	\end{equation*}
	Thus, we obtain the inequality
	\begin{equation}\label{LLU4}
		\begin{split}
			&s^{-1}\iint_{Q} \xi^{-1} (|u_t |^2 + \left|\Div(|x|^\alpha A\nabla u)\right|^2) dx dt \\
			&+C\iint_Q s^3 \lambda^4 \xi^3\left||x|^\alpha A \nabla \eta \cdot \nabla \eta \right|^2|u|^2 dx dt  + C s  \lambda^2 \iint_Q\xi\left||x|^\alpha A  \nabla \eta \cdot \nabla u \right|^2 dx  d t\\
			&\qquad\le C\left\|e^{-s \sigma} f\right\|^2
			+Cs \lambda \iint_Q \xi |x|^\alpha A \nabla u \cdot \nabla udx  d t+Cs^3 \lambda^3 \int_0^T \int_{\omega_{0}} \xi^3|u|^2dx  d t.
		\end{split}
	\end{equation}
	Note that
	\begin{equation*}
		\begin{split}
			&Cs^3\lambda^4\iint_Q \xi^3 \left| |x|^\alpha A \nabla  \eta \cdot \nabla \eta\right| ^2 u^2 dxdt\\
			&\qquad\ge Cs^3\lambda^4\int_0^T \int_{\Om\backslash\om} \xi^3  u^2 dxdt + Cs^3\lambda^4\int_0^T \int_{\om} \xi^3 \left| |x|^\alpha A \nabla  \eta \cdot \nabla \eta\right| ^2 u^2 dxdt\\
			&\qquad\ge Cs^3\lambda^4\iint_Q \xi^3  u^2 dxdt -Cs^3\lambda^4\int_0^T \int_{\om} \xi^3  u^2 dxdt.
		\end{split}
	\end{equation*}
	Hence, \eqref{LLU4}  can be written as
	\begin{equation*}
		\begin{split}
			&s^{-1}\iint_{Q} \xi^{-1} (|u_t |^2 + \left|\Div(|x|^\alpha A\nabla u)\right|^2) dx dt \\
			&+C\iint_Q s^3 \lambda^4 \xi^3|u|^2 dx dt  + C s  \lambda^2 \iint_Q\xi\left||x|^\alpha A  \nabla \eta \cdot \nabla u \right|^2 dx  d t\\
			&\qquad\le C\left\|e^{-s \sigma} f\right\|^2
			+Cs \lambda \iint_Q \xi |x|^\alpha A \nabla u \cdot \nabla udx  d t+Cs^3 \lambda^4 \int_0^T \int_{\omega_{0}} \xi^3|u|^2dx  d t.
		\end{split}
	\end{equation*}
	Considering the term $s \lambda \iint_Q \xi |x|^\alpha A \nabla u \cdot \nabla udx d t$, we have
	\begin{equation*}
		\begin{split}
			s \lambda \iint_Q \xi |x|^\alpha A \nabla u \cdot \nabla udx  d t=& s \lambda \iint_\Sigma \xi u |x|^\alpha A \nabla u \cdot n d s d t - s \lambda \iint_Q \Div(\xi |x|^\alpha A\nabla u) udx  d t\\
			=&- s \lambda \iint_Q \Div(\xi |x|^\alpha A\nabla u) udx  d t.
		\end{split}
	\end{equation*}
	Then, we can further get
	\begin{equation*}
		\begin{split}
			&- s \lambda \iint_Q \Div(\xi |x|^\alpha A\nabla u) udx  d t\\
			&\qquad=- s \lambda^2 \iint_Q \xi u\nabla \eta \cdot |x|^\alpha A\nabla u dx  d t - s \lambda \iint_Q \xi\Div( |x|^\alpha A\nabla u) udx  d t\\
			&\qquad\le C \lambda^2 \iint_Q \xi \left||x|^\alpha A  \nabla \eta \cdot \nabla u \right| ^2dx  d t + Cs^2 \lambda^2 \iint_Q \xi u^2dx  d t\\
			&\qquad\qquad + C s^{-1} \lambda^{-1} \iint_Q \xi^{-1} \left| \Div( |x|^\alpha A\nabla u)\right| ^2dx  d t + Cs^3 \lambda^3 \iint_Q \xi^3 u^2dx  d t.
		\end{split}
	\end{equation*}
	Thus, the term $s \lambda \iint_Q \xi |x|^\alpha A \nabla u \cdot \nabla udx d t$ can be absorbed by the other terms. This implies that
	\begin{equation}\label{LLU5}
		\begin{split}
			&s^{-1}\iint_{Q} \xi^{-1} (|u_t |^2 + \left|\Div(|x|^\alpha A\nabla u)\right|^2) dx dt \\
			&+C s^3 \lambda^4 \iint_Q\xi^3|u|^2 dx dt  + C s  \lambda^2 \iint_Q\xi\left||x|^\alpha A  \nabla \eta \cdot \nabla u \right|^2 dx  d t\\
			&\qquad\le C\left\|e^{-s \sigma} f\right\|^2
			+Cs^3 \lambda^4 \int_0^T \int_{\omega_{0}} \xi^3|u|^2dx  d t.
		\end{split}
	\end{equation}
	By employing classical arguments, we can then revert back to the original variable $w$ and conclude the result.
	$\hfill\blacksquare$
	\subsection{The case of \texorpdfstring{$0 \notin \omega_0 $}{}}
	
	\hspace*{\fill}\\
	
	Unlike in case $0 \in \omega_0 $, in case $0 \notin \omega_0 $, we can no longer employ the old method. Therefore, we have adopted a new approach that involves approximating with $\epsilon$, allowing us to control the degenerate part by the non-degenerate part. This is also why we ultimately achieve approximate controllability.
	
	%
	
	In what follows, $C>0$ represents a generic constant which depends $\epsilon$ (we should note that $\epsilon$ depends on the solution $w$ of equation \eqref{3.1}).
	Using standard arguments, we can assume that $w$ possesses sufficient regularity. Specifically, we have $w \in H^1(0,T; \mathcal{H}_0^1(\Omega))$ and $w \in L^2(0,T; H^2(\Om^\epsilon))$.
	\begin{lemma}\label{ETA2}
		For $\epsilon>0$ small enough, there is a weight function that satisfies the following conditions:\\
		($i$)
		$|x|^\alpha A \nabla \eta$, $|x|^\alpha A \nabla \eta \cdot \nabla \eta$, $\nabla(|x|^\alpha A \nabla \eta \cdot \nabla \eta)$, $\Div(|x|^\alpha A \nabla \eta)$, $\Div(|x|^\alpha A \nabla \eta\cdot \nabla \eta)$, $\left(  A\nabla  \eta\right)_i \frac{\partial(|x|^\alpha A)}{\partial x_i} \\
		\in C^0(\overline{{\Omega}})$.\\
		($ii$) 	
		$| |x|^\alpha A \nabla \eta \cdot \nabla \eta |^2 \ge C >0 \ \mbox{in} \ \Om\backslash(\omega\cup \Om_\epsilon)$.
		Moreover, we have
		$\frac{\partial \eta}{\partial {n}} <0 \mbox{ on } \Gamma$, 
		where ${n}$ is the unit outer normal vector on $\Gamma=\partial\Om$.
	\end{lemma}
	\begin{proof}
		Consider a function $\zeta \in C^2(\overline{\Omega})$ satisfying $\zeta = 0$ on $\partial \Om$, $\zeta>0$ in $\Om$,  and $\nabla \zeta \ge C>0$ in $\Om\backslash\omega$.
		Define a smooth function (by standard mollification method)
		\begin{equation*}
			h(x)=
			\begin{cases}
				2\epsilon^{-2}|x|^2, & |x|\leq \frac{1}{2}\epsilon,\\
				1, & |x|\geq \epsilon,\\
				\mbox{smooth the function}\ \epsilon^{-1}|x|, & \mbox{else}. 
			\end{cases}
		\end{equation*}
		Let $\eta=h\cdot \zeta$, then $\eta$ satisfies item ($i$) and item ($ii$). 
		
		Finally, we have $\frac{\partial \eta}{\partial {n}} <0 \mbox{ on } \Gamma$ since $\eta>0$ in $\Om^\epsilon$ and $\eta=0$ on $\Gamma$.
	\end{proof}
	\begin{lemma}\label{EPSILON}
		For every $u\in \mathcal{H}_0^1(\Om)$, there is a sufficiently small $\epsilon=\epsilon(u)$ such that 
		\begin{equation}\label{epsilonU}
			\int_{0}^{T}\int_{\Omega_\epsilon}   u^2 dx  dt\le \frac{1}{4}  \int_{0}^{T}\int_{\Om\backslash(\om\cup \Omega_\epsilon)}   u^2 dx dt
		\end{equation}
		and
		\begin{equation}\label{epsilon2U}
			\int_{0}^{T}\int_{\Omega_\epsilon}   |x|^\alpha A \nabla u \cdot \nabla u dx  dt\le \frac{1}{4}  \int_{0}^{T}\int_{\Om\backslash(\om\cup \Omega_\epsilon)}   |x|^\alpha A \nabla u \cdot \nabla u dx  dt.
		\end{equation}
	\end{lemma}
	\begin{proof}
		Let's first prove that equation \eqref{epsilonU} holds. If $u=0$ in $\Om\backslash\om$, then equation \eqref{epsilonU} holds. If $u\neq0$ in $\Om\backslash\om$, then we can make $\epsilon$ small enough so that
		\begin{equation*}
			\begin{split}
				\int_{0}^{T}\int_{\Omega_\epsilon}   u^2 dx  dt\le \frac{1}{4}\int_{0}^{T}\int_{\Om\backslash(\om\cup \Omega_\epsilon)}   u^2 dx  dt.
			\end{split}
		\end{equation*}
		The same applies to equation \eqref{epsilon2U}.
	\end{proof}
	
	\begin{remark}
		If we are given a function $u\in \mathcal{H}_0^1(\Om)$, we can choose a suitable constant $\epsilon>0$ such that Lemma \ref{ETA2} and Lemma \ref{EPSILON} hold. It is important to note that $\epsilon$ is determined by $u$ in this context. We first specify $u$ and then choose an appropriate $\epsilon$ to ensure the validity of the lemma mentioned above. This is also why, in the subsequent proofs, we will derive conclusions regarding the dependence of the coefficient $C$ on $u$.
	\end{remark}
	
	Now, for any $\lambda>\lambda_0>0$ and $s>s_0>0$, let us introduce the following functions and
	constants:
	\begin{equation*}
		\begin{split}
			& \theta(t):=[t(T-t)]^{-4}, \quad \xi(x, t):=\theta(t) e^{ \lambda(8|\eta|_\infty+\eta (x))}, \quad \sigma(x, t):=\theta(t) e^{10 \lambda|\eta|_\infty}-\xi(x, t).
		\end{split}
	\end{equation*}
	Next, let us establish the Carleman estimate for case $0 \notin \omega_0$.
	
	{\bf Proof of Theorem \ref{Carleman2}:}
	Consider $s>s_0>0$ and introduce
	\begin{equation*}
		u=e^{-s\sigma} w.  
	\end{equation*}
	Then, the following properties hold for $u$:
	\begin{itemize}
		\item [($i$)] $u=\frac{\partial u}{\partial x_i}=0$ at $t=0$ and $t=T$;
		\item [($ii$)] $u=0$  on $\Sigma$;
		\item [($iii$)] Let 
		\begin{equation*}
			\begin{split}
				P_1 u
				&:=-2s\lambda^2\xi u |x|^\alpha A \nabla \eta\cdot \nabla \eta- 2s\lambda\xi  |x|^\alpha A \nabla u\cdot \nabla \eta +u_t=:(P_1u)_1+(P_1u)_2+(P_1u)_3,\\
				P_2 u
				&:=s^2\lambda^2\xi^2 u |x|^\alpha A \nabla \eta\cdot \nabla \eta+\Div(|x|^\alpha A \nabla u)+s \sigma_t u=:(P_2u)_1+(P_2u)_2+(P_2u)_3,\\
				f_{s,\lambda}
				&:=e^{-s \sigma} f - s\lambda^2\xi u |x|^\alpha A \nabla \eta\cdot \nabla \eta + s\lambda \xi u\Div(|x|^\alpha A \nabla \eta) .
			\end{split}
		\end{equation*}
		Then 
		\begin{equation*}
			P_1 u+P_2 u=f_{s,\lambda}.
		\end{equation*}
	\end{itemize}
	From ($iii$), it follows that
	\begin{equation}\label{PPP}
		\left\|P_1 u\right\|_{L^2(Q)}^2+\left\|P_2 u\right\|_{L^2(Q)}^2+2\sum_{i,j=1}^{3}\left((P_1 u)_i, (P_2 u)_j\right)_{L^2(Q)}=\left\|f_{s,\lambda}\right\|_{L^2(Q)}^2.
	\end{equation}
	
	First, we have
	\begin{equation*}
		\left((P_1 u)_1, (P_2 u)_1\right)_{L^2(Q)}=-2s^3 \lambda^4 \iint_Q \left| |x|^\alpha A \nabla  \eta \cdot \nabla \eta\right| ^2 \xi^3 u^2 dx  dt=:A,
	\end{equation*}
	and
	\begin{equation*}
		\begin{split}
			\left((P_1 u)_2, (P_2 u)_1\right)_{L^2(Q)}=
			&s^3 \lambda^3 \iint_Q \xi^3\Div( |x|^{\alpha} A \nabla  \eta)   |x|^{\alpha} A \nabla  \eta \cdot \nabla \eta  u^2 dx  dt\\
			&+3s^3 \lambda^4 \iint_Q \left| |x|^\alpha A \nabla  \eta \cdot \nabla \eta\right| ^2 \xi^3 u^2 dx  dt\\
			&+s^3 \lambda^3 \iint_Q \xi^3 |x|^\alpha A \nabla  \eta   \cdot \nabla( |x|^{\alpha} A \nabla  \eta\cdot \nabla  \eta)u^2 dx  dt\\
			=&: B_1 + B_2 +B_3.
		\end{split}
	\end{equation*}
	We clearly have that $A+B_2=s^3 \lambda^4 \iint_Q \left| |x|^\alpha A \nabla  \eta \cdot \nabla \eta\right| ^2 \xi^3 u^2 dx  dt$ is a positive term. As a consequence of the properties of $\eta$ (see Lemma \ref{ETA2}), we have
	\begin{equation*}
		\begin{split}
			A+B_2=&s^3 \lambda^4 \iint_Q \left| |x|^\alpha A \nabla  \eta \cdot \nabla \eta\right| ^2 \xi^3 u^2 dx  dt\\
			\ge& s^3 \lambda^4 \iint_Q \left| |x|^\alpha A \nabla  \eta \cdot \nabla \eta\right| ^2 \xi^3 u^2 dx  dt - s^3 \lambda^4 \int_{0}^{T} \int_{\om}   \xi^3 u^2 dx  dt=: A^*+B^*.
		\end{split}
	\end{equation*}
	For some $C=C(\om_0,\Om)$, and
	\begin{equation*}
		\begin{split}
			B_1=&s^3 \lambda^3 \iint_Q \xi^3\Div( |x|^{\alpha} A \nabla  \eta)   |x|^{\alpha} A \nabla  \eta \cdot \nabla \eta  u^2 dx  dt\\
			=&s^3 \lambda^3 \int_{0}^{T}\int_{\om} \xi^3\Div( |x|^{\alpha} A \nabla  \eta)   |x|^{\alpha} A \nabla  \eta \cdot \nabla \eta  u^2 dx  dt\\
			&+s^3 \lambda^3 \int_{0}^{T}\int_{\Om\backslash\om} \xi^3\Div( |x|^{\alpha} A \nabla  \eta)   |x|^{\alpha} A \nabla  \eta \cdot \nabla \eta  u^2 dx  dt\\
			\ge&-Cs^3 \lambda^3 \int_{0}^{T}\int_{\om} \xi^3  u^2 dx  dt -Cs^3 \lambda^3 \int_{0}^{T}\int_{\Om\backslash\om} \xi^3  u^2 dx  dt,\\
		\end{split}
	\end{equation*}
	where
	\begin{equation*}
		\begin{split}
			-Cs^3 \lambda^3 \int_{0}^{T}\int_{\Om\backslash\om} \xi^3  u^2 dx  dt=
			-Cs^3 \lambda^3 \int_{0}^{T}\int_{\Om\backslash(\om\cup \Omega_\epsilon)} \xi^3  u^2 dx  dt -Cs^3 \lambda^3 \int_{0}^{T}\int_{\Omega_\epsilon} \xi^3  u^2 dx  dt.
		\end{split}
	\end{equation*}
	From Lemma \ref{EPSILON}, we can deduce
	\begin{equation}\label{epsilon}
		\begin{split}
			-Cs^3 \lambda^3 \int_{0}^{T}\int_{\Omega_\epsilon} \xi^3  u^2 dx  dt\ge -\frac{1}{4}Cs^3 \lambda^3 \int_{0}^{T}\int_{\Om\backslash(\om\cup \Omega_\epsilon)} \xi^3  u^2 dx  dt,
		\end{split}
	\end{equation}
	for some $C=C(\om_0,\Om,u)$, hence we have
	\begin{equation*}
		\begin{split}
			-Cs^3 \lambda^3 \int_{0}^{T}\int_{\Om\backslash\om} \xi^3  u^2 dx  dt\ge
			-Cs^3 \lambda^3 \int_{0}^{T}\int_{\Om\backslash(\om\cup \Omega_\epsilon)} \xi^3  u^2 dx  dt,
		\end{split}
	\end{equation*}
	and since $| |x|^\alpha A \nabla \eta \cdot \nabla \eta |^2 \ge C >0 \ \mbox{in} \ \Om\backslash(\omega\cup \Om_\epsilon)$, we can deduce 
	\begin{equation*}
		\begin{split}
			-Cs^3 \lambda^3 \int_{0}^{T}\int_{\Om\backslash(\om\cup \Omega_\epsilon)} \xi^3  u^2 dx  dt\ge&
			-Cs^3 \lambda^3 \int_{0}^{T}\int_{\Om\backslash(\om\cup \Omega_\epsilon)} \xi^3 | |x|^\alpha A \nabla \eta \cdot \nabla \eta |^2 u^2 dx  dt\\
			\ge&
			-Cs^3 \lambda^3 \iint_{Q} \xi^3 | |x|^\alpha A \nabla \eta \cdot \nabla \eta |^2 u^2 dx  dt.
		\end{split}
	\end{equation*}
	Thus we have
	\begin{equation*}
		\begin{split}
			B_1
			\ge-Cs^3 \lambda^3 \int_{0}^{T}\int_{\om} \xi^3  u^2 dx  dt 	-Cs^3 \lambda^3 \iint_{Q} \xi^3 | |x|^\alpha A \nabla \eta \cdot \nabla \eta |^2 u^2 dx  dt.
		\end{split}
	\end{equation*}
	In a similar way, we can deduce
	\begin{equation*}
		\begin{split}
			B_3=&s^3 \lambda^3 \iint_Q \xi^3 |x|^\alpha A \nabla  \eta   \Div( |x|^{\alpha} A \nabla  \eta\cdot \nabla  \eta)u^2 dx  dt\\
			\ge&-Cs^3 \lambda^3 \int_{0}^{T}\int_{\om} \xi^3  u^2 dx  dt 	-Cs^3 \lambda^3 \iint_{Q} \xi^3 | |x|^\alpha A \nabla \eta \cdot \nabla \eta |^2 u^2 dx  dt.
		\end{split}
	\end{equation*}
	The terms $B_1$ and $B_3$ are absorbed by $A^*$ and $B^*$ by simply taking $\lambda \ge C$.
	
	By the definitions of $\xi$, we have $\xi\xi_t\le C\xi^3$. From item ($i$) and  item ($ii$), we have
	\begin{equation*}
		\begin{split}
			\left((P_1 u)_3, (P_2 u)_1\right)_{L^2(Q)}=&
			s^2 \lambda^2 \iint_Q \xi^2 |x|^\alpha A \nabla \eta \cdot \nabla \eta u u_t dx  dt\\
			=&\frac{1}{2}s^2 \lambda^2 \int_\Om \xi^2 |x|^\alpha A \nabla \eta \cdot \nabla \eta u^2 dx \bigg|_0^T - s^2 \lambda^2 \iint_Q \xi \xi_t |x|^\alpha A \nabla \eta \cdot \nabla \eta u^2 dx  dt\\
			=&- s^2 \lambda^2 \iint_Q \xi \xi_t |x|^\alpha A \nabla \eta \cdot \nabla \eta u^2 dx  dt\\
			\ge&- Cs^2 \lambda^2 \iint_Q \xi^3 | |x|^\alpha A \nabla \eta \cdot \nabla \eta |^2 u^2 dx  dt
			- Cs^2 \lambda^2 \int_{0}^{T}\int_{\om} \xi^3  u^2 dx  dt,
		\end{split}
	\end{equation*}
	for some $C=C(\om_0,\Om,u)$. This term can also be absorbed by $A^*$ and $B^*$. Consequently, we have
	\begin{equation}\label{P1P21}
		\begin{split}
			\left((P_1 u), (P_2 u)_1\right)_{L^2(Q)}=&\left((P_1 u)_1, (P_2 u)_1\right)_{L^2(Q)} + \left((P_1 u)_2, (P_2 u)_1\right)_{L^2(Q)}+\left((P_1 u)_3, (P_2 u)_1\right)_{L^2(Q)}\\
			\ge&C s^3 \lambda^4 \iint_Q \left| |x|^\alpha A \nabla  \eta \cdot \nabla \eta\right| ^2 \xi^3 u^2 dx  dt - Cs^3 \lambda^4 \int_{0}^{T} \int_{\om}   \xi^3 u^2 dx  dt.
		\end{split}
	\end{equation}
	On the other hand, we have
	\begin{equation*}
		\begin{split}
			\left((P_1 u)_1, (P_2 u)_2\right)_{L^2(Q)}=&-2s \lambda^2 \iint_Q \xi  |x|^\alpha A \nabla \eta \cdot \nabla \eta  \Div(|x|^\alpha A \nabla u) u dx  dt\\
			=&2s \lambda^2 \iint_Q \xi |x|^\alpha A \nabla \eta \cdot \nabla \eta |x|^\alpha A \nabla u \cdot \nabla u   dx  dt\\
			&+2s \lambda^2 \iint_Q \xi  u |x|^\alpha A \nabla u  \cdot \nabla( |x|^\alpha A \nabla \eta \cdot \nabla \eta)dx  dt\\
			&+2s \lambda^3 \iint_Q \xi u |x|^\alpha A \nabla \eta \cdot \nabla \eta |x|^\alpha A \nabla u \cdot \nabla \eta   dx  dt\\
			=&:C_1+C_2+C_3.
		\end{split}
	\end{equation*}
	We will keep $C_1$ on the left-hand side. For $C_2$ and $C_3$, similar to the estimate of $B_1$, from Lemma \ref{ETA2} and Lemma \ref{EPSILON}, we have
	\begin{equation*}
		\begin{split}
			C_2=&2s \lambda^2 \iint_Q \xi \nabla( |x|^\alpha A \nabla \eta \cdot \nabla \eta) u |x|^\alpha A \nabla u  dx  dt\\
			\ge&-Cs^2 \lambda^3 \iint_Q \xi |x|^\alpha A \nabla( |x|^\alpha A \nabla \eta \cdot \nabla \eta) \cdot \nabla( |x|^\alpha A \nabla \eta \cdot \nabla \eta) u^2  dx  dt\\
			&-C \lambda \iint_Q \xi |x|^\alpha A \nabla u \cdot \nabla u dx  dt\\
			\ge& -Cs^2 \lambda^3 \iint_Q \left| |x|^\alpha A \nabla  \eta \cdot \nabla \eta\right| ^2 \xi u^2 dx  dt - Cs^2 \lambda^3 \int_{0}^{T} \int_{\om}   \xi u^2 dx  dt\\
			&-C \lambda \iint_Q \xi |x|^\alpha A \nabla \eta \cdot \nabla \eta|x|^\alpha A \nabla u \cdot \nabla u dx  dt
			-C \lambda \int_{0}^{T}\int_\om \xi |x|^\alpha A \nabla u \cdot \nabla u dx  dt.
		\end{split}
	\end{equation*}
	For $C_3$ we have
	\begin{equation*}
		\begin{split}
			C_3=&2s \lambda^3 \iint_Q \xi u |x|^\alpha A \nabla \eta \cdot \nabla \eta |x|^\alpha A \nabla u \cdot \nabla \eta   dx  dt\\
			\ge&-Cs^2 \lambda^4 \iint_Q \xi \left| |x|^\alpha A \nabla  \eta \cdot \nabla \eta\right| ^2 u^2  dx  dt\\
			&-C \lambda^2 \iint_Q \xi  |x|^\alpha A \nabla \eta \cdot \nabla \eta |x|^\alpha A \nabla u \cdot \nabla u dx  dt.
		\end{split}
	\end{equation*}
	Thus, we have 
	\begin{equation*}
		\begin{split}
			&\left((P_1 u)_1, (P_2 u)_2\right)_{L^2(Q)}\\
			&\qquad\ge C_1
			-Cs^2 \lambda^4 \iint_Q \left| |x|^\alpha A \nabla  \eta \cdot \nabla \eta\right| ^2 \xi u^2 dx  dt - Cs^2 \lambda^3 \int_{0}^{T} \int_{\om}   \xi u^2 dx  dt\\
			&\qquad\qquad -C \lambda^2 \iint_Q \xi  |x|^\alpha A \nabla \eta \cdot \nabla \eta |x|^\alpha A \nabla u \cdot \nabla u dx  dt
			-C \lambda \int_{0}^{T}\int_\om \xi |x|^\alpha A \nabla u \cdot \nabla u dx  dt.
		\end{split}
	\end{equation*}
	We also have
	\begin{equation*}
		\begin{split}
			\left((P_1 u)_2, (P_2 u)_2\right)_{L^2(Q)}
			=& -2s \lambda \iint_Q \xi  |x|^\alpha A \nabla  u \cdot \nabla \eta 
			\Div(|x|^\alpha A \nabla u) dx dt\\
			=& - 2s \lambda \iint_\Sigma \xi \left[|x|^\alpha A\nabla u\cdot n\right]^2\frac{\partial\eta}{\partial n} ds dt\\
			&+2s \lambda^2 \iint_Q \xi  |x|^\alpha A \nabla  u \cdot \nabla \eta  |x|^\alpha A \nabla  u \cdot \nabla \eta dx dt\\
			&+2s \lambda \iint_Q \xi 
			|x|^\alpha A \nabla u \cdot \nabla(  |x|^\alpha A \nabla  u \cdot \nabla \eta) dx dt\\
			=&:D_1+D_2+D_3.
		\end{split}
	\end{equation*}
	Let us remark that $D_2$ is a positive term.
	Furthermore,
	\begin{equation*}
		\begin{split}
			D_3=&2s \lambda \iint_Q \xi  |x|^\alpha A \nabla u \cdot \nabla(  |x|^\alpha A \nabla  u \cdot \nabla \eta)dx dt\\
			=&2s \lambda \sum_{i=1}^{N}\iint_Q  \xi \left( |x|^\alpha A \nabla  \eta\right)_i |x|^\alpha A \nabla  u \cdot \frac{\partial}{\partial x_i} (\nabla u)  + \xi(\nabla u)_i  |x|^\alpha A \nabla  u \cdot \nabla \left( |x|^\alpha A \nabla  \eta\right) _i  dx dt\\
			=&:D_{31}+D_{32}.
		\end{split}
	\end{equation*}
	After some additional computations we also see that
	\begin{equation*}
		\begin{split}
			D_{31}=&2s \lambda \sum_{i=1}^{N}\iint_Q  \xi \left( |x|^\alpha A \nabla  \eta\right)_i |x|^\alpha A \nabla  u \cdot \frac{\partial}{\partial x_i} (\nabla u) dx dt\\
			=& s \lambda \sum_{i=1}^{N}\iint_Q  \xi \left( |x|^\alpha A \nabla  \eta\right)_i \left[ \frac{\partial}{\partial x_i} \left(|x|^\alpha A \nabla  u \cdot \nabla u  \right) - \frac{\partial(|x|^\alpha A)}{\partial x_i} \nabla  u \cdot \nabla u \right]  dx dt\\
			=&s \lambda \iint_\Sigma \xi [|x|^\alpha An\cdot n][|x|^\alpha A\nabla u\cdot \nabla u]\frac{\partial \eta}{\partial n} ds dt
			-s \lambda \iint_Q \xi  |x|^\alpha A \nabla u \cdot \nabla u \Div(|x|^\alpha A \nabla \eta)  dx dt\\
			&-s \lambda \sum_{i=1}^{N}\iint_Q \xi \left( |x|^\alpha A \nabla  \eta\right)_i \frac{\partial(|x|^\alpha A)}{\partial x_i} \nabla  u \cdot \nabla u   dx dt - s\lambda^2 \iint_Q \xi  |x|^\alpha A \nabla \eta \cdot \nabla \eta |x|^\alpha A \nabla u \cdot \nabla u dx  dt\\
			=&:D_{311}+D_{312}+D_{313} +D_{314}.
		\end{split}
	\end{equation*}
	It is easy to see that $C_1 + D_{314}=s\lambda^2 \iint_Q \xi  |x|^\alpha A \nabla \eta \cdot \nabla \eta |x|^\alpha A \nabla u \cdot \nabla u dx  dt\ge 0$. By virtue of the properties satisfied by $\eta$ and $\nabla u=\frac{\partial u}{\partial n}n$ on $\Sigma$, we notice that $D_1 + D_{311}\ge0$, and for $D_{32}$, we have
	\begin{equation*}
		\begin{split}
			D_{32}=&2s \lambda \sum_{i=1}^{N}\iint_Q \xi (\nabla u)_i  |x|^\alpha A \nabla  u \cdot \nabla \left( |x|^\alpha A \nabla  \eta\right) _i  dx dt\\
			=&2s \lambda \sum_{i=1}^{N} \iint_Q (\xi^{\frac{1}{2}} |x|^{\frac{\alpha}{2}} A^{\frac{1}{2}} \nabla u)(\xi^{\frac{1}{2}} |x|^{\frac{\alpha}{2}} A^{\frac{1}{2}} \nabla \left( |x|^\alpha A \nabla  \eta\right) _i  (\nabla u)_i ) dx dt\\
			\ge& -Cs \lambda \iint_Q \xi  |x|^\alpha A\nabla u  \cdot \nabla u dx dt\\
			&-Cs \lambda \sum_{i=1}^{N}\iint_Q \xi  |x|^\alpha A\nabla \left( |x|^\alpha A \nabla  \eta\right) _i \cdot \nabla \left( |x|^\alpha A \nabla  \eta\right) _i  (\nabla u)_i \cdot (\nabla u)_i dx dt\\
			\ge& -Cs \lambda \iint_Q \xi  |x|^\alpha A\nabla u  \cdot \nabla u dx dt
			-Cs \lambda \sum_{i=1}^{N}\iint_Q \xi  |x|^\alpha   (\nabla u)_i \cdot (\nabla u)_i dx dt\\
			\ge& -Cs \lambda \iint_Q \xi  |x|^\alpha A\nabla u  \cdot \nabla u dx dt
			-Cs \lambda \iint_Q \xi  |x|^\alpha A\nabla u  \cdot \nabla u dx dt\\
			\ge&-Cs \lambda \iint_Q \xi  |x|^\alpha A\nabla u  \cdot \nabla u dx dt\\
			\ge&-C s\lambda \iint_Q \xi  |x|^\alpha A \nabla \eta \cdot \nabla \eta |x|^\alpha A \nabla u \cdot \nabla u dx  dt -Cs \lambda \int_{0}^{T}\int_\om \xi  |x|^\alpha A\nabla u  \cdot \nabla u dx dt.
		\end{split}
	\end{equation*}
	In a similar way, we also have 
	\begin{equation*}
		\begin{split}
			D_{312}=&-s \lambda \iint_Q \xi  |x|^\alpha A \nabla u \cdot \nabla u \Div(|x|^\alpha A \nabla \eta)  dx dt\\
			\ge&-Cs \lambda \iint_Q \xi  |x|^\alpha A\nabla u  \cdot \nabla u dx dt\\
			\ge&-C s\lambda \iint_Q \xi  |x|^\alpha A \nabla \eta \cdot \nabla \eta |x|^\alpha A \nabla u \cdot \nabla u dx  dt -Cs \lambda \int_{0}^{T}\int_\om \xi  |x|^\alpha A\nabla u  \cdot \nabla u dx dt,
		\end{split}
	\end{equation*}
	and
	\begin{equation*}
		\begin{split}
			D_{313}=&-s \lambda \sum_{i=1}^{N}\iint_Q \xi \left( |x|^\alpha A \nabla  \eta\right)_i \frac{\partial(|x|^\alpha A)}{\partial x_i} \nabla  u \cdot \nabla u   dx dt\\
			\ge&-Cs \lambda \iint_Q \xi  |x|^\alpha A\nabla u  \cdot \nabla u dx dt\\
			\ge&-C s\lambda \iint_Q \xi  |x|^\alpha A \nabla \eta \cdot \nabla \eta |x|^\alpha A \nabla u \cdot \nabla u dx  dt -Cs \lambda \int_{0}^{T}\int_\om \xi  |x|^\alpha A\nabla u  \cdot \nabla u dx dt.
		\end{split}
	\end{equation*}
	Consequently,
	\begin{equation*}
		\begin{split}
			\left((P_1 u)_2, (P_2 u)_2\right)_{L^2(Q)}
			\ge&
			D_{314} + Cs \lambda^2 \iint_Q \xi  \left| |x|^\alpha A \nabla  u \cdot \nabla \eta\right| ^2  dx  dt\\
			&-C s\lambda \iint_Q \xi  |x|^\alpha A \nabla \eta \cdot \nabla \eta |x|^\alpha A \nabla u \cdot \nabla u dx  dt\\
			&-Cs \lambda \int_{0}^{T}\int_\om \xi  |x|^\alpha A\nabla u  \cdot \nabla u dx dt.
		\end{split}
	\end{equation*}
	Additionally, we find that
	\begin{equation*}\label{P13P22}
		\begin{split}
			\left((P_1 u)_3, (P_2 u)_2\right)_{L^2(Q)}
			=
			& \iint_Q \Div(|x|^\alpha A \nabla  u) u_t dx  dt\\
			=&\iint_\Sigma u_t |x|^\alpha A\nabla u \cdot n ds dt - \iint_{Q} |x|^\alpha A\nabla u \cdot \nabla u_t dxdt\\
			=&-\frac{1}{2}\int_\Om |x|^\alpha A\nabla u \cdot \nabla u dx \bigg|_0^T =0.
		\end{split}
	\end{equation*}
	Thus, for sufficiently large $\lambda$, we have
	\begin{equation}\label{P1P22}
		\begin{split}
			&\left((P_1 u), (P_2 u)_2\right)_{L^2(Q)}\\
			&\qquad= \left((P_1 u)_1, (P_2 u)_2\right)_{L^2(Q)} + \left((P_1 u)_2, (P_2 u)_2\right)_{L^2(Q)}+\left((P_1 u)_3, (P_2 u)_2\right)_{L^2(Q)}\\
			&\qquad\ge Cs \lambda^2 \iint_Q \xi  \left| |x|^\alpha A \nabla  u \cdot \nabla \eta\right| ^2  dx  dt
			+C s\lambda^2 \iint_Q \xi  |x|^\alpha A \nabla \eta \cdot \nabla \eta |x|^\alpha A \nabla u \cdot \nabla u dx  dt\\
			&\qquad\qquad-Cs^2 \lambda^4 \iint_Q \left| |x|^\alpha A \nabla  \eta \cdot \nabla \eta\right| ^2 \xi u^2 dx  dt
			- Cs^2 \lambda^3 \int_{0}^{T} \int_{\om}   \xi u^2 dx  dt\\
			&\qquad\qquad-Cs \lambda \int_{0}^{T}\int_\om \xi  |x|^\alpha A\nabla u  \cdot \nabla u dx dt.
		\end{split}
	\end{equation}
	Let us now consider the scalar product, since $\sigma_t\xi \le C\xi^3$, we have
	\begin{equation*}
		\begin{split}
			\left((P_1 u)_1, (P_2 u)_3\right)_{L^2(Q)}
			=&-2s^2\lambda^2 \iint_Q \xi \sigma_t |x|^\alpha A \nabla  \eta \cdot \nabla \eta u^2 dxdt\\
			\ge&-Cs^2\lambda^2 T \iint_Q \xi^3 u^2 dxdt\\
			\ge&-Cs^2\lambda^2 T \int_0^T\int_{\omega} \xi^3 u^2 dxdt
			-Cs^2 \lambda^2 T\iint_Q \left| |x|^\alpha A \nabla  \eta \cdot \nabla \eta\right| ^2 \xi^3 u^2 dx  dt.
		\end{split}
	\end{equation*}
	Obviously, this is absorbed by $A^*$ and $B^*$ if we choose $s$ and $\lambda$ to be sufficiently large.
	Furthermore,
	\begin{equation*}
		\begin{split}
			\left((P_1 u)_2, (P_2 u)_3\right)_{L^2(Q)}
			=&-2s^2\lambda \iint_Q \xi \sigma_t |x|^\alpha A \nabla  \eta \cdot \nabla  u u dxdt\\
			=& s^2\lambda^2 \iint_Q \xi \sigma_t |x|^\alpha A \nabla  \eta \cdot \nabla  \eta u^2 dxdt
			+s^2\lambda \iint_Q \xi  |x|^\alpha A \nabla  \eta  \cdot \nabla \sigma_t u^2 dxdt\\
			&+s^2\lambda \iint_Q \xi  \sigma_t \Div(|x|^\alpha A \nabla  \eta)  u^2 dxdt\\
			=&:E_1 +E_2 +E_3.
		\end{split}
	\end{equation*}
	Using the same way as the previous proof, we have
	\begin{equation*}
		\begin{split}
			E_1
			=& s^2\lambda^2 \iint_Q \xi \sigma_t |x|^\alpha A \nabla  \eta \cdot \nabla  \eta u^2 dxdt\\
			\ge&-Cs^2\lambda^2 T\int_{0}^{T}\int_\om \xi^3 u^2 dxdt
			-Cs^2\lambda^2 T\iint_Q \xi^3 \left| |x|^\alpha A \nabla  \eta \cdot \nabla \eta\right| ^2 u^2 dxdt.
		\end{split}
	\end{equation*}
	For $E_2$ and $E_3$ we have,
	\begin{equation*}
		\begin{split}
			E_2
			=&s^2\lambda \iint_Q \xi \nabla \sigma_t |x|^\alpha A \nabla  \eta  u^2 dxdt\\
			\ge&-Cs^2\lambda^2 T\int_{0}^{T}\int_\om \xi^3 u^2 dxdt
			-Cs^2\lambda^2 T\iint_Q \xi^3 \left| |x|^\alpha A \nabla  \eta \cdot \nabla \eta\right| ^2 u^2 dxdt,
		\end{split}
	\end{equation*}
	and
	\begin{equation*}
		\begin{split}
			E_3
			=&s^2\lambda \iint_Q \xi  \sigma_t \Div(|x|^\alpha A \nabla  \eta)  u^2 dxdt\\
			\ge&-Cs^2\lambda T\int_{0}^{T}\int_\om \xi^3 u^2 dxdt
			-Cs^2\lambda T\iint_Q \xi^3 \left| |x|^\alpha A \nabla  \eta \cdot \nabla \eta\right| ^2 u^2 dxdt.
		\end{split}
	\end{equation*}
	Thus, we have
	\begin{equation*}
		\begin{split}
			\left((P_1 u)_2, (P_2 u)_3\right)_{L^2(Q)}
			\ge&-Cs^2\lambda^2 T\int_{0}^{T}\int_\om \xi^3 u^2 dxdt
			-Cs^2\lambda^2 T\iint_Q \xi^3 \left| |x|^\alpha A \nabla  \eta \cdot \nabla \eta\right| ^2 u^2 dxdt.
		\end{split}
	\end{equation*}
	Finally, we have
	\begin{equation*}
		\begin{split}
			\left((P_1 u)_3, (P_2 u)_3\right)_{L^2(Q)}
			=&s\iint_{Q}\sigma_t u_t u dx dt = -\frac{1}{2} s \iint_{Q}\sigma_{tt} u^2 dx dt\ge-C s T^2\iint_{Q}\xi^3 u^2 dx dt\\
			\ge&-Cs T^2\int_{0}^{T}\int_\om \xi^3 u^2 dxdt
			-Cs T^2\iint_Q \xi^3 \left| |x|^\alpha A \nabla  \eta \cdot \nabla \eta\right| ^2 u^2 dxdt,
		\end{split}
	\end{equation*}
	since $\sigma_{t t}\le \xi^3$.
	
	Now, we can deduce for sufficiently large $\lambda$ and $s$ that
	\begin{equation}\label{P1P23}
		\begin{split}
			\left((P_1 u), (P_2 u)_3\right)_{L^2(Q)}=&\left((P_1 u)_1, (P_2 u)_3\right)_{L^2(Q)} + \left((P_1 u)_2, (P_2 u)_3\right)_{L^2(Q)}+\left((P_1 u)_3, (P_2 u)_3\right)_{L^2(Q)}\\
			\ge&-Cs^2\lambda^3\int_{0}^{T}\int_\om \xi^3 u^2 dxdt
			-Cs^2\lambda^3\iint_Q \xi^3 \left| |x|^\alpha A \nabla  \eta \cdot \nabla \eta\right| ^2 u^2 dxdt.
		\end{split}
	\end{equation}
	Taking into account \eqref{P1P21}, \eqref{P1P22}, and \eqref{P1P23}, we obtain
	\begin{equation*}
		\begin{split}
			\left( P_1 u, P_2 u \right)_{L^2(Q)}
			\ge&Cs^3\lambda^4\iint_{Q} \xi^3 \left| |x|^\alpha A \nabla  \eta \cdot \nabla \eta\right| ^2 u^2 dxdt
			+Cs \lambda^2 \iint_Q \xi  \left| |x|^\alpha A \nabla  u \cdot \nabla \eta\right| ^2  dx  dt\\
			&+C s\lambda^2 \iint_Q \xi  |x|^\alpha A \nabla \eta \cdot \nabla \eta |x|^\alpha A \nabla u \cdot \nabla u dx  dt
			-Cs^3\lambda^4\int_{0}^{T}\int_\om \xi^3 u^2 dxdt\\
			&-Cs \lambda \int_{0}^{T}\int_\om \xi  |x|^\alpha A\nabla u  \cdot \nabla u dx dt.
		\end{split}
	\end{equation*}
	From item ($iii$), for sufficiently large $s$ and $\lambda$, we have
	\begin{equation*}
		\begin{split}
			&\left\| P_1 u \right\| ^2_{L^2(Q)} + \left\| P_2 u \right\| ^2_{L^2(Q)}
			+Cs^3\lambda^4\iint_Q \xi^3 \left| |x|^\alpha A \nabla  \eta \cdot \nabla \eta\right| ^2 u^2 dxdt\\
			&+Cs \lambda^2 \iint_Q \xi  \left| |x|^\alpha A \nabla  u \cdot \nabla \eta\right| ^2  dx  dt
			+C s\lambda^2 \iint_Q \xi  |x|^\alpha A \nabla \eta \cdot \nabla \eta |x|^\alpha A \nabla u \cdot \nabla u dx  dt\\
			&\qquad\le \left\| f_{s,\lambda} \right\| ^2_{L^2(Q)}+ Cs^3\lambda^4\int_{0}^{T}\int_\om \xi^3 u^2 dxdt	
			+Cs \lambda \int_{0}^{T}\int_\om \xi  |x|^\alpha A\nabla u  \cdot \nabla u dx dt\\
			&\qquad\le \left\| e^{-s\sigma} f \right\| ^2_{L^2(Q)}
			+ s^2\lambda^4\iint_{Q} \xi^3 \left| |x|^\alpha A \nabla  \eta \cdot \nabla \eta\right| ^2 u^2 dxdt
			+s^2\lambda^2\iint_{Q} \xi^3 \left| \Div(|x|^\alpha A \nabla  \eta)\right| ^2 u^2 dxdt\\
			&\qquad\qquad+ Cs^3\lambda^4\int_{0}^{T}\int_\om \xi^3 u^2 dxdt	
			+Cs \lambda \int_{0}^{T}\int_\om \xi  |x|^\alpha A\nabla u  \cdot \nabla u dx dt.
		\end{split}
	\end{equation*}
	Thus, we also have
	\begin{equation}\label{P1P2}
		\begin{split}
			&\left\| P_1 u \right\| ^2_{L^2(Q)} + \left\| P_2 u \right\| ^2_{L^2(Q)}
			+Cs^3\lambda^4\iint_Q \xi^3 \left| |x|^\alpha A \nabla  \eta \cdot \nabla \eta\right| ^2 u^2 dxdt\\
			&+Cs \lambda^2 \iint_Q \xi  \left| |x|^\alpha A \nabla  u \cdot \nabla \eta\right| ^2  dx  dt
			+C s\lambda^2 \iint_Q \xi  |x|^\alpha A \nabla \eta \cdot \nabla \eta |x|^\alpha A \nabla u \cdot \nabla u dx  dt\\
			&\qquad\le \left\| e^{-s\sigma} f \right\| ^2_{L^2(Q)}
			+ Cs^3\lambda^4\int_{0}^{T}\int_\om \xi^3 u^2 dxdt	
			+Cs \lambda \int_{0}^{T}\int_\om \xi  |x|^\alpha A\nabla u  \cdot \nabla u dx dt.
		\end{split}
	\end{equation}
	The final step will be to add integrals of $\left|  \Div(|x|^\alpha A \nabla  u)\right|  ^2$ and $\left|  u_t\right|  ^2$ to the left-hand side of \eqref{P1P2}. This can be made using the expressions
	of $P_1 u$ and $P_2 u$. Indeed, from item $(iii)$ we have
	\begin{equation*}
		\begin{split}
			&s^{-1} \iint_Q \xi^{-1}\left|u_t\right|^2dx  d t\\
			&\qquad=s^{-1} \iint_Q \xi^{-1}(P_1 u + 2s \lambda^2 \xi u |x|^\alpha A \nabla  \eta \cdot \nabla \eta   + 2s \lambda \xi  |x|^\alpha A \nabla  u \cdot \nabla \eta)^2dx  d t\\
			&\qquad\le Cs^{-1}\left\|P_1 u\right\|^2+ C s\lambda^4\iint_Q \xi \left| |x|^\alpha A \nabla  \eta \cdot \nabla \eta\right| ^2 u^2 dxdt
			+C s\lambda^2\iint_Q \xi \left| |x|^\alpha A \nabla  u \cdot \nabla \eta\right| ^2 dxdt,
		\end{split}
	\end{equation*}
	and
	\begin{equation*}
		\begin{split}
			&s^{-1}   \iint_Q \xi^{-1}\left|\Div(|x|^\alpha A\nabla u)\right|^2dx  d t\\
			&\qquad= s^{-1}   \iint_Q \xi^{-1}(P_2 u -s^2\lambda^2\xi^2 u |x|^\alpha A \nabla \eta \cdot  \nabla \eta - s \sigma_t u)^2dx  d t\\
			&\qquad\le C s^{-1} \left\|P_2 u\right\|^2+C s^3 \lambda^4   \iint_Q \xi^3| |x|^\alpha A \nabla \eta\cdot \nabla \eta |^2|u|^2dx  d t +C sT^2   \iint_Q \xi^{3}|u|^2dx  d t.
		\end{split}
	\end{equation*}
	Accordingly, we deduce from \eqref{P1P2} that
	\begin{equation}\label{P1UP2U}
		\begin{split}
			&Cs^{-1} \iint_Q \xi^{-1}\left|u_t\right|^2dx  d t +Cs^{-1}   \iint_Q \xi^{-1}\left|\Div(|x|^\alpha A\nabla u)\right|^2dx  d t\\
			&+Cs^3\lambda^4\iint_Q \xi^3 \left| |x|^\alpha A \nabla  \eta \cdot \nabla \eta\right| ^2 u^2 dxdt
			+Cs \lambda^2 \iint_Q \xi  \left| |x|^\alpha A \nabla  u \cdot \nabla \eta\right| ^2  dx  dt\\
			&+C s\lambda^2 \iint_Q \xi |x|^\alpha A \nabla \eta \cdot \nabla \eta |x|^\alpha A \nabla u \cdot \nabla u dx  dt\\
			&\qquad\le \left\| e^{-s\sigma} f \right\| ^2_{L^2(Q)}
			+ Cs^3\lambda^4\int_{0}^{T}\int_\om \xi^3 u^2 dxdt	
			+Cs \lambda \int_{0}^{T}\int_\om \xi  |x|^\alpha A\nabla u  \cdot \nabla u dx dt,
		\end{split}
	\end{equation}
	for sufficiently large $s$ and $\lambda$.
	It is worth noting that, by Lemma \ref{ETA2} and Lemma \ref{EPSILON},  we have
	\begin{equation*}
		\begin{split}
			&C s\lambda^2 \iint_Q \xi  |x|^\alpha A \nabla \eta \cdot \nabla \eta |x|^\alpha A \nabla u \cdot \nabla u dx  dt\\
			&\qquad=C s\lambda^2 \int_0^T \int_{\Om\backslash(\Om\cup \Om_\epsilon)} \xi  |x|^\alpha A \nabla \eta \cdot \nabla \eta |x|^\alpha A \nabla u \cdot \nabla u dx  dt\\
			&\qquad\qquad+ C s\lambda^2 \int_0^T \int_{\Om_\epsilon} \xi  |x|^\alpha A \nabla \eta \cdot \nabla \eta |x|^\alpha A \nabla u \cdot \nabla u dx  dt\\
			&\qquad\qquad+ C s\lambda^2 \int_0^T \int_{\om} \xi  |x|^\alpha A \nabla \eta \cdot \nabla \eta |x|^\alpha A \nabla u \cdot \nabla u dx  dt\\
			&\qquad\ge C s\lambda^2 \int_0^T \int_{\Om\backslash(\Om\cup \Om_\epsilon)} \xi  |x|^\alpha A \nabla u \cdot \nabla u dx  dt
			+ C s\lambda^2 \int_0^T \int_{\om} \xi  |x|^\alpha A \nabla \eta \cdot \nabla \eta |x|^\alpha A \nabla u \cdot \nabla u dx  dt\\
			&\qquad\ge \frac{1}{2}C s\lambda^2 \int_0^T \int_{\Om\backslash(\Om\cup \Om_\epsilon)} \xi  |x|^\alpha A \nabla u \cdot \nabla u dx  dt
			+ \frac{1}{2}C s\lambda^2 \int_0^T \int_{\Om_\epsilon} \xi   |x|^\alpha A \nabla u \cdot \nabla u dx  dt\\
			&\qquad\qquad +C s\lambda^2 \int_0^T \int_{\om} \xi  |x|^\alpha A \nabla \eta \cdot \nabla \eta |x|^\alpha A \nabla u \cdot \nabla u dx  dt\\
			&\qquad\ge C s\lambda^2 \int_0^T \int_{\Om\backslash\om} \xi   |x|^\alpha A \nabla u \cdot \nabla u dx  dt + C s\lambda^2 \int_0^T \int_{\om} \xi  |x|^\alpha A \nabla \eta \cdot \nabla \eta |x|^\alpha A \nabla u \cdot \nabla u dx  dt\\
			&\qquad\ge C s\lambda^2 \iint_Q \xi   |x|^\alpha A \nabla u \cdot \nabla u dx  dt - C s\lambda^2 \int_0^T \int_{\om} \xi  |x|^\alpha A \nabla u \cdot \nabla u dx  dt,
		\end{split}
	\end{equation*}
	and
	\begin{equation*}
		\begin{split}
			&Cs^3\lambda^4\iint_Q \xi^3 \left| |x|^\alpha A \nabla  \eta \cdot \nabla \eta\right| ^2 u^2 dxdt
			\ge Cs^3\lambda^4\iint_Q \xi^3  u^2 dxdt -Cs^3\lambda^4\int_0^T \int_{\om} \xi^3  u^2 dxdt.
		\end{split}
	\end{equation*}
	Hence, \eqref{P1UP2U}  can be written as
	\begin{equation}
		\begin{split}
			&Cs^{-1} \iint_Q \xi^{-1}\left|u_t\right|^2dx  d t +Cs^{-1}   \iint_Q \xi^{-1}\left|\Div(|x|^\alpha A\nabla u)\right|^2dx  d t\\
			&+Cs^3\lambda^4\iint_Q \xi^3  u^2 dxdt
			+Cs \lambda^2 \iint_Q \xi  \left| |x|^\alpha A \nabla  u \cdot \nabla \eta\right| ^2  dx  dt\\
			&+C s\lambda^2 \iint_Q \xi  |x|^\alpha A \nabla u \cdot \nabla u dx  dt\\
			&\qquad\le \left\| e^{-s\sigma} f \right\| ^2_{L^2(Q)}
			+ Cs^3\lambda^4\int_{0}^{T}\int_\om \xi^3 u^2 dxdt	
			+Cs \lambda^2 \int_{0}^{T}\int_\om \xi  |x|^\alpha A\nabla u  \cdot \nabla u dx dt,
		\end{split}
	\end{equation}
	We are now ready to eliminate the second integral on the right-hand side. To this end, let us introduce a function $\gamma=\gamma(x)$, with
	\begin{equation*}
		\gamma\in C_0^2(\omega_0), \quad 0\le \gamma\le 1,\quad \gamma\equiv 1 \mbox{ in } \omega, \quad \gamma\equiv 0 \mbox{ outside } \om+\Om_\epsilon=\{x+y\mid x\in\om, y\in\Om_\epsilon\}, 
	\end{equation*}
	and let us make some computations. We have
	\begin{equation*}
		\begin{split}
			s \lambda^2 \int_{0}^{T}\int_\om \xi |x|^\alpha A \nabla u \cdot \nabla udx  d t\le s \lambda^2 \int_{0}^{T}\int_{\om_0} \gamma \xi |x|^\alpha A \nabla u \cdot \nabla udx  d t,
		\end{split}
	\end{equation*}
	and
	\begin{equation*}
		\begin{split}
			&s \lambda^2 \int_{0}^{T}\int_{\om_0}\gamma \xi |x|^\alpha A \nabla u \cdot \nabla udx  d t\\
			&\qquad=- s \lambda^2 \int_{0}^{T}\int_{\om_0} \Div(\gamma\xi |x|^\alpha A\nabla u) udx  d t\\
			&\qquad=- s \lambda^3 \int_{0}^{T}\int_{\om_0} \gamma\xi u |x|^\alpha A\nabla u \cdot \nabla \eta dx  d t - s \lambda^2 \int_{0}^{T}\int_{\om_0}\gamma \xi\Div( |x|^\alpha A\nabla u) udx  d t\\
			&\qquad\qquad- s \lambda^2 \int_{0}^{T}\int_{\om_0}  \xi |x|^\alpha A\nabla u\cdot\nabla\gamma udx  d t
			\\
			&\qquad\le C \lambda^2 \int_{0}^{T}\int_{\om_0} \xi \left| |x|^\alpha A \nabla  u \cdot \nabla \eta\right| ^2 dx  d t + Cs^2 \lambda^4 \int_{0}^{T}\int_{\om_0} \xi u^2dx  d t\\
			&\qquad\qquad+C \delta s^{-1}  \int_{0}^{T}\int_{\om_0} \xi^{-1} \left| \Div( |x|^\alpha A\nabla u)\right| ^2dx  d t + C \frac{1}{\delta}s^3 \lambda^4 \int_{0}^{T}\int_{\om_0} \xi^3 u^2dx  d t\\
			&\qquad\qquad+C \int_{0}^{T}\int_{\om_0}  |x|^\alpha A \nabla u \cdot \nabla u  dx  d t + Cs^2 \lambda^4 \int_{0}^{T}\int_{\om_0}  \xi^2 u^2 dx  d t.
		\end{split}
	\end{equation*}
	Thus, for $\delta$ that is small enough, the term $s \lambda^2 \int_{0}^{T}\int_{\om} \xi |x|^\alpha A \nabla u \cdot \nabla udx d t$ can be absorbed by the other terms. 
	
	This implies that for sufficiently large $s$ and $\lambda$, we have
	\begin{equation*}
		\begin{split}
			&Cs^{-1} \iint_Q \xi^{-1}\left|u_t\right|^2dx  d t +Cs^{-1}   \iint_Q \xi^{-1}\left|\Div(|x|^\alpha A\nabla u)\right|^2dx  d t\\
			&+Cs^3\lambda^4\iint_Q \xi^3  u^2 dxdt
			+Cs \lambda^2 \iint_Q \xi  \left| |x|^\alpha A \nabla  u \cdot \nabla \eta\right| ^2  dx  dt\\
			&+C s\lambda^2 \iint_Q \xi   |x|^\alpha A \nabla u \cdot \nabla u dx  dt\\
			&\qquad\le \left\| e^{-s\sigma} f \right\| ^2_{L^2(Q)}
			+ Cs^3\lambda^4\int_{0}^{T}\int_{\om_0} \xi^3 u^2 dxdt.
		\end{split}
	\end{equation*}
	By employing classical arguments, we can then revert back to the original variable $w$ and conclude the result.
	$\hfill\blacksquare$
	
	\section{Appendix}				
	{\bf Proof of Theorem \ref{b1}:}
	Assuming  $ z_{m} $  has the structure \eqref{a5}.  Since  $ \{w_{k}\}^{\infty}_{k=1} $ is an orthogonal basis of  $ \mcH_0^1(\Omega) $ and  $ L^{2}(\Omega) $,  we obtain
	\begin{equation*}\label{a9}
		\left(\frac{d}{dt}z_{m}(t), w_{k}\right)=\frac{d}{dt}d_{m}^{k}(t). 
	\end{equation*}
	Furthermore, 
	\begin{equation*}\label{a10}
		B[z_{m}, w_{k};t]=\sum_{l=1}^{m}e^{kl}(t)d_{m}^{l}(t), 
	\end{equation*}
	for  $ e^{kl}(t):=B[w_{l}, w_{k};t],  k, l=1, \cdots,  m$.  Let us further write  $ g^{k}(t):=(\chi_{\omega_0}g, w_{k}),  k=1, \cdots, m$.  Then \eqref{a7} becomes the linear system of ODEs
	\begin{equation}\label{a11}
		\frac{d}{dt}d_{m}^{k}(t)+\sum_{l=1}^{m}e^{kl}(t)d_{m}^{l}(t)=g^{k}(t),  \quad k=1, \cdots, m, 
	\end{equation}
	subject to the initial conditions \eqref{a6}.  Since the standard existence theory for ordinary differential equations,  there exists a unique absolutely continuous function  $ d_{m}(t)=(d_{m}^{1}(t), \cdots, d_{m}^{m}(t)) $  satisfying \eqref{a6} and \eqref{a11} for  $t\in(0, T)$ a.e.  And then  $ z_{m} $  defined by \eqref{a5} solves \eqref{a7} for   $t\in(0, T)$ a.e.
	$\hfill\blacksquare$
	
	{\bf Proof of Theorem \ref{b2}:}
	Multiplying equation \eqref{a7} by  $ \frac{d}{dt}d_{m}^{k}(t) $,  sum for $k=1,\cdots, m$,  and then recall \eqref{a5} to find
	\begin{equation}\label{a13}
		\left(\frac{d}{dt}z_{m}, z_{m}\right)+B[z_{m}, z_{m};t] =(\chi_{\omega_0}g, z_{m}),
	\end{equation}
	for  $0\leq t \leq T $ a.e. Since
	\begin{equation*}\label{a14}
		B[z_{m}, z_{m};t]=\int_{\Omega} |x|^\alpha A\nabla z_{m}\cdot \nabla  z_{m} dx,
	\end{equation*}
	and
	\begin{equation*}
		\|z_{m}\|_{\mcH_0^1(\Omega)}=\left(\int_{\Omega} |x|^\alpha A\nabla z_{m}\cdot \nabla  z_{m} dx\right)^{\frac{1}{2}}. 
	\end{equation*}	
	This implies
	\begin{equation}\label{a15}
		\begin{split}
			B[z_{m}, z_{m};t]
			=\|z_{m}\|_{\mcH_0^1(\Omega)}^{2},
		\end{split}
	\end{equation}
	and for a.e.\! $t\in(0, T)$ and all $m\in\mathbb{N}$,
	\begin{equation*}\label{a16}
		|(\chi_{\omega_0}g, z_{m})| \leq \frac{1}{2}\|\chi_{\omega_0}g \|_{L^{2}(\Omega)}^2+\frac{1}{2}\|z_{m} \|_{L^{2}(\Omega)}^2,  
	\end{equation*}
	and for a.e.\! $t\in (0,T)$ we have
	\begin{equation*}\label{a17}
		\left(\frac{d}{dt}z_{m}, z_{m}\right)=\frac{d}{dt}\left(\frac{1}{2}\|z_{m} \|_{L^{2}(\Omega)}^2\right).
	\end{equation*} 
	Consequently \eqref{a13} yields the inequality 
	\begin{equation}\label{a18}
		\frac{d}{dt}\left(\frac{1}{2}\|z_{m} \|_{L^{2}(\Omega)}^2\right) + \|z_{m}\|_{\mcH_0^1(\Omega)}^2 \leq C\|z_{m} \|_{L^{2}(\Omega)}^2 + C\|\chi_{\omega_0}g \|_{L^{2}(\Omega)}^2
	\end{equation}
	for a.e.\! $t\in (0,T)$, hereinafter, we denote $C>0$ are different constants by context. 
	
	Now writting
	\begin{alignat}{2}
		\eta(t)
		:=\|z_{m}(t) \|_{L^{2}(\Omega)}^2, \quad 
		\xi(t)
		:=\|\chi_{\omega_0}g(t) \|_{L^{2}(\Omega)}^2.  \nonumber
	\end{alignat} 
	Then \eqref{a18} imples 
	\begin{equation*}\label{a21}
		\frac{d}{dt}\eta(t)\leq C\eta(t)+C\xi(t)
	\end{equation*}
	for a.e.\! $t\in(0, T)$. Thus the differential form of Gronwall's inequality yields the estimate
	\begin{equation}\label{a22}
		\eta(t)\leq e^{Ct}\eta(0)+C\int_{0}^{t}\xi(s)  ds \mbox{ for all } t\in(0, T). 
	\end{equation} 
	Since  $ \eta(0)=\|z_{m}(0) \|_{L^{2}(\Omega)} \leq \|z_{0} \|_{L^{2}(\Omega)} $  by \eqref{a6},  we obtain from  \eqref{a22} the estimate 
	\begin{equation}\label{a23}
		\max\limits_{t\in(0, T)} \|z_{m}(t) \|_{L^{2}(\Omega)} \leq C \left(\|\chi_{\omega_0}g \|_{L^{2}(0, T;L^{2}(\Omega))}+\|z_{0} \|_{L^{2}(\Omega)}\right). 
	\end{equation}
	
	Returning once more to inequality \eqref{a18},  we integrate from $0$ to $T$ and employ the inequality above to find 
	\begin{equation}\label{a24}
		\|z_{m} \|_{L^{2}(0, T;\mcH_0^1(\Omega))}^2= \int_{0}^{T} \|z_{m} \|_{\mcH_0^1(\Omega)}^2 dt \leq   C \left(\|\chi_{\omega_0}g \|_{L^{2}(0, T;L^{2}(\Omega))}^2+\|z_{0} \|_{L^{2}(\Omega)}^2\right). 
	\end{equation}
	
	Fix any  $ v\in \mcH_0^1(\Omega) $ with  $ \|v\|_{\mcH_0^1(\Omega)}\leq 1 $,  and write  $ v=v^{1}+v^{2} $,  where  $ v^{1}\in {\rm span}\{w_{k}\}^{m}_{k=1} $  and  $ (v^{2}, w_{k})=0,  k=1, \cdots,  m$. 
	Since $ \{w_{k}\}^{\infty}_{k=1} $  are orthogonal in $ \mcH_0^1(\Omega)$,  we have $\|v^{1}\|_{\mcH_0^1(\Omega)} \leq \|v\|_{\mcH_0^1(\Omega)}\leq 1 $. 
	Utiliuing \eqref{a7},  we deduce for  $t\in(0, T)$ a.e. \!  that
	\begin{equation*}\label{a25}
		\left(\frac{d}{dt}z_{m}, v^{1}\right)+B[z_{m}, v^{1};t] =(\chi_{\omega_0}g, v^{1}). 
	\end{equation*} 
	Then \eqref{a5} implies that
	\begin{equation*}\label{a26}
		\left\langle \frac{d}{dt}z_{m}, v \right\rangle_{\left\langle \mathcal{H}^{-1}(\Omega),\mathcal{H}_0^{1}(\Omega)\right\rangle }=\left( \frac{d}{dt}z_{m}, v \right)  =\left(\frac{d}{dt}z_{m}, v^{1}\right)=(\chi_{\omega_0}g, v^{1})-B[z_{m}, v^{1};t]. 
	\end{equation*}
	Consequently,
	\begin{equation*}\label{a27}
		\left|\left(\frac{d}{dt}z_{m}, v\right)\right|\leq C \left(\|\chi_{\omega_0}g \|_{L^{2}(\Omega)}+\|z_{m}\|_{\mcH_0^1(\Omega)}\right), 
	\end{equation*}
	since  $ \|v^{1}\|_{\mcH_0^1(\Omega)} \leq 1 $ and \eqref{a15}.  
	Thus
	\begin{equation*}\label{a28}
		\left\|\frac{d}{dt}z_{m}\right\|_{\mathcal{H}^{-1}(\Omega)} \leq C \left(\|\chi_{\omega_0}g \|_{L^{2}(\Omega)}+\|z_{m}\|_{\mcH_0^1(\Omega)}\right), 
	\end{equation*}
	and therefore
	\begin{equation}\label{a29}
		\begin{split}
			\int_{0}^{T} \left\|\frac{d}{dt}z_{m}\right\|_{\mcH^{-1}(\Omega)}^2  dt \le& \int_{0}^{T} \left(\|\chi_{\omega_0}g \|_{L^{2}(\Omega)}^{2}+\|z_{m}\|_{\mcH_0^1(\Omega)}^{2}\right)  dt\\
			\le&
			C \left(\|\chi_{\omega_0}g \|_{L^{2}(0, T;L^{2}(\Omega))}+\|z_{0} \|_{L^{2}(\Omega)}\right).
		\end{split}
	\end{equation}
	
	Finally,  combining \eqref{a23},  \eqref{a24} and \eqref{a29},  we obtain \eqref{a12}.  
	$\hfill\blacksquare$
	\begin{lemma}\label{b9}
		Assume
		\begin{equation*}
			\begin{cases}
				z_{k} \rightharpoonup z ,  & \mbox{in } {L^{2}(0,T;\mathcal{H}_0^{1}(\Omega))}; \\
				\frac{d}{dt}z_{k} \rightharpoonup  v ,  & \mbox{in } {L^{2}(0,T;\mathcal{H}^{-1}(\Omega))}. 
			\end{cases}
		\end{equation*}
		We have $\frac{d}{dt}z=v $ in ${L^{2}(0,T;\mathcal{H}^{-1}(\Omega))}$.
	\end{lemma}
	\begin{proof}
		Claim: for all $\phi\in C^{\infty}_{0}(0,T)$, and for each $w\in \mathcal{H}_0^{1}(\Omega)$, we have
		\begin{equation*}
			\left\langle\int_{0}^{T}\left(\frac{d}{dt}\phi(t)\right) z(t)  dt,w\right\rangle_{\left\langle \mathcal{H}^{-1}(\Omega),\mathcal{H}_0^{1}(\Omega)\right\rangle }=\left\langle-\int_{0}^{T}v(t)\phi(t) dt,w\right\rangle_{\left\langle \mathcal{H}^{-1}(\Omega),\mathcal{H}_0^{1}(\Omega)\right\rangle }.
		\end{equation*}
		If the claim is proved, we have
		\begin{equation*}
			\int_{0}^{T}\left(\frac{d}{dt}\phi(t)\right) z(t)  dt=-\int_{0}^{T}v(t)\phi(t)   dt\ \mbox{ in }\ \mathcal{H}^{-1}(\Omega).
		\end{equation*}
		By the definition of weak derivative, we can imply
		\begin{equation*}
			\int_{0}^{T}\left(\frac{d}{dt}\phi(t)\right) z(t)  dt=-\int_{0}^{T}\left(\frac{d}{dt}z(t)\right)\phi(t)   dt \ \mbox{ in }\ \mathcal{H}^{-1}(\Omega).
		\end{equation*}
		Then we have $\frac{d}{dt}z=v $ in ${L^{2}(0,T;\mathcal{H}^{-1}(\Omega))}$.
		So we just need to prove the claim. 
		
		Notice that
		\begin{equation*}
			\begin{split}
				\left\langle\int_{0}^{T}\left(\frac{d}{dt}\phi(t)\right) z(t)  dt,w\right\rangle_{\left\langle \mathcal{H}^{-1}(\Omega),\mathcal{H}_0^{1}(\Omega)\right\rangle }
				&=\int_{0}^{T}\left\langle\left(\frac{d}{dt}\phi(t)\right) z(t),w \right\rangle_{\left\langle \mathcal{H}^{-1}(\Omega),\mathcal{H}_0^{1}(\Omega)\right\rangle }  dt\\
				&=\int_{0}^{T} \left\langle z(t),\left(\frac{d}{dt}\phi(t)\right)w \right\rangle_{\left\langle \mathcal{H}^{-1}(\Omega),\mathcal{H}_0^{1}(\Omega)\right\rangle } \ dt.
			\end{split}
		\end{equation*}
		Since $z_{k} \rightharpoonup z$ in $  {L^{2}(0,T;\mathcal{H}_0^{1})}$, we have
		\begin{equation*}
			\begin{split}
				&\int_{0}^{T} \left\langle z(t),\left(\frac{d}{dt}\phi(t)\right)w\right\rangle_{\left\langle \mathcal{H}^{-1}(\Omega),\mathcal{H}_0^{1}(\Omega)\right\rangle } dt\\
				&\qquad=\lim\limits_{k \rightarrow \infty} \int_{0}^{T} \left\langle z_{k}(t),\left(\frac{d}{dt}\phi(t)\right)w \right\rangle_{\left\langle \mathcal{H}^{-1}(\Omega),\mathcal{H}_0^{1}(\Omega)\right\rangle }  dt\\
				&\qquad=\lim\limits_{k \rightarrow \infty} \int_{0}^{T} \left\langle z_{k}(t)\frac{d}{dt}\phi(t)  ,w \right\rangle_{\left\langle \mathcal{H}^{-1}(\Omega),\mathcal{H}_0^{1}(\Omega)\right\rangle }  dt\\
				&\qquad=\lim\limits_{k \rightarrow \infty} \left\langle \int_{0}^{T} z_{k}(t)\frac{d}{dt}\phi(t)  dt  ,w \right\rangle_{\left\langle \mathcal{H}^{-1}(\Omega),\mathcal{H}_0^{1}(\Omega)\right\rangle }\\
				&\qquad=-\lim\limits_{k \rightarrow \infty} \left\langle \int_{0}^{T} \left(\frac{d}{dt}z_{k}(t)\right)\phi(t) \ dt  ,w \right\rangle_{\left\langle \mathcal{H}^{-1}(\Omega),\mathcal{H}_0^{1}(\Omega)\right\rangle } \\
				&\qquad=-\lim\limits_{k \rightarrow \infty} \int_{0}^{T} \left\langle \left(\frac{d}{dt}z_{k}(t)\right)\phi(t)  ,w \right\rangle_{\left\langle \mathcal{H}^{-1}(\Omega),\mathcal{H}_0^{1}(\Omega)\right\rangle } dt.
			\end{split}
		\end{equation*}
		Since $\frac{d}{dt}z_{k} \rightharpoonup v$ in $  {L^{2}(0,T;\mathcal{H}^{1})}$, we have
		\begin{equation*}
			\begin{split}
				&-\lim\limits_{k \rightarrow \infty} \int_{0}^{T} \left\langle  \left(\frac{d}{dt}z_{k}(t)\right)\phi(t)  ,w \right\rangle_{\left\langle \mathcal{H}^{-1}(\Omega),\mathcal{H}_0^{1}(\Omega)\right\rangle } dt\\
				&\qquad=-\int_{0}^{T} \left\langle v(t)\phi(t)  ,w \right\rangle_{\left\langle \mathcal{H}^{-1}(\Omega),\mathcal{H}_0^{1}(\Omega)\right\rangle } dt=\left\langle-\int_{0}^{T}v(t)\phi(t) dt, w\right\rangle_{\left\langle \mathcal{H}^{-1}(\Omega),\mathcal{H}_0^{1}(\Omega)\right\rangle }.
			\end{split}
		\end{equation*} 
		Hence
		\begin{equation*}
			\left\langle\int_{0}^{T}\left(\frac{d}{dt}\phi(t) \right)z(t)  dt,w\right\rangle_{\left\langle \mathcal{H}^{-1}(\Omega),\mathcal{H}_0^{1}(\Omega)\right\rangle }=\left\langle-\int_{0}^{T}v(t)\phi(t)  dt,w\right\rangle_{\left\langle \mathcal{H}^{-1}(\Omega),\mathcal{H}_0^{1}(\Omega)\right\rangle }.
		\end{equation*}
		Then the claim is proved.
	\end{proof}
	{\bf Proof of Theorem \ref{b3}:}
	Since the energy estimates \eqref{a12},  we see that the sequence  $ \{z_{m}\}^{\infty}_{m=1} $ 	is bounded in  $ {L^{2}(0,T;\mathcal{H}_0^{1}(\Omega))} $,  and  $ \{\frac{d}{dt}z_{m}\}^{\infty}_{m=1} $  is bounded in  $ {L^{2}(0,T;\mathcal{H}^{-1}(\Om))} $.  From Lemma \ref{b9}, consequently there exists a subsequence  $ \{z_{m_l}\}^{\infty}_{l=1} \subset \{z_{m}\}^{\infty}_{m=1} $  and a function  $ z \in L^{2}(0,T;\mathcal{H}_0^{1}(\Om)) $,  with  $ \frac{d}{dt}z \in L^{2}(0,T;\mathcal{H}^{-1}(\Om)) $,  such that
	\begin{equation}\label{a30}
		\begin{cases}
			z_{m_l} \rightharpoonup z ;  & \mbox{in } {L^{2}(0,T;\mathcal{H}_0^{1}(\Om))}, \\
			\frac{d}{dt}z_{m_l} \rightharpoonup \frac{d}{dt}z ,  & \mbox{in } {L^{2}(0,T;\mathcal{H}^{-1}(\Om))}. 
		\end{cases}
	\end{equation}
	
	Next,  fixing an integer $N\in\mathbb{N}$ and choosing a function  $ v\in C^{1}([0,T];\mathcal{H}_0^{1}) $  with the form
	\begin{equation}\label{a31}
		v(t):=\sum_{k=1}^{N}d_{m}^{k}(t)w_{k}, 
	\end{equation}
	where  $ \{d^{k}(t)\}^{N}_{k=1}$  are given smooth functions in \eqref{a11}.  We choose  $ m\ge N $,  multiplying \eqref{a7} by   $ d_{m}^{k}(t) $,  sum $k=1, \cdots, N$,  and then integrate with respect to $t$ to find
	\begin{equation}\label{a32}
		\int_{0}^{T}\left[\left\langle \frac{d}{dt}z_{m}, v \right\rangle_{\left\langle \mcH^{-1}(\Om), \mcH_0^1(\Om)\right\rangle } +B[z_{m}, v;t]\right] dt =\int_{0}^{T} (\chi_{\omega_0}g, v) dt. 
	\end{equation} 
	We set  $ m=m_l $  and recall \eqref{a30},  to find upon passing to weak limits that
	\begin{equation}\label{a33}
		\int_{0}^{T}\left[\left\langle \frac{d}{dt}z, v \right\rangle_{\left\langle \mcH^{-1}(\Om), \mcH_0^1(\Om)\right\rangle } +B[z, v;t] \right]dt =\int_{0}^{T} (\chi_{\omega_0}g, v) dt. 
	\end{equation} 
	This equality then holds for all functions  $ v\in L^{2}(0,T;\mathcal{H}_0^{1}(\Om)) $ since the functions of the form \eqref{a31} are dense in this space.  Hence,  in particular, 
	\begin{equation}\label{a34}
		\left\langle \frac{d}{dt}z, v \right\rangle_{\left\langle \mcH^{-1}(\Om), \mcH_0^1(\Om)\right\rangle } +B[z, v;t] =(\chi_{\omega_0}g, v),  
	\end{equation}
	for each $v\in \mathcal{H}_0^{1}(\Om)$ and a.e.\! $0 \le t\le T$. From Theorem 3 (page 303) in \cite{Evans} we see that furthermore $z \in C([0,T];L^{2}(\Omega))$. 
	
	In order to prove  $ z(0)=z_{0}$,  we first note from \eqref{a33} that  for 
	\begin{equation}\label{a35}
		\int_{0}^{T} \left[-\left\langle \frac{d}{dt}v, z \right\rangle_{\left\langle \mcH^{-1}(\Om), \mcH_0^1(\Om)\right\rangle } +B[z, v;t]\right] dt-(z(0), v(0)) =\int_{0}^{T} (\chi_{\omega_0}g, v) dt,
	\end{equation}
	for each  $ v\in C^{1}([0,T];\mathcal{H}_0^{1}(\Om)) $  with  $ v(T)=0 $. 
	
	Similarly,  from \eqref{a32} we deduce
	\begin{equation}\label{a36}
		\int_{0}^{T} \left[-\left\langle \frac{d}{dt}v, z_{m} \right\rangle_{\left\langle \mcH^{-1}(\Om), \mcH_0^1(\Om)\right\rangle } +B[z_{m}, v;t] \right]dt-(z_{m}(0), v(0)) =\int_{0}^{T} (\chi_{\omega_0}g, v) dt. 
	\end{equation}
	
	We set $m=m_l$ and once again employ \eqref{a30} to find
	\begin{equation}\label{a37}
		\int_{0}^{T} \left[-\left\langle \frac{d}{dt}v, z \right\rangle_{\left\langle \mcH^{-1}(\Om), \mcH_0^1(\Om)\right\rangle } +B[z, v;t] \right]dt -(z_{0}, v(0))=\int_{0}^{T} (\chi_{\omega_0}g, v) dt, 
	\end{equation}
	since  $ z_{m_l}(0) \to z_{0} $  in  $ L^{2}(\Omega) $  by \eqref{a5} and \eqref{a6}.  As  $ v(0) $  is arbitrary,  comparing \eqref{a35} and \eqref{a37},  we conclude  $ z(0)=z_{0} $. 
	$\hfill\blacksquare$

	\vspace{3mm}
	
	\noindent{\bf Acknowledgement}
	
	\vspace{2mm}
	
	This work is supported by the National Natural Science Foundation of China, the Natural Sciences and Engineering Research Council of Canada (RGPIN-2018-05687), and a centennial fund of the University of Alberta.

	
	\bibliographystyle{abbrv}

	\bibliography{ref20230701}
	%
	%
	%
\end{document}